\documentclass[reqno,12pt]{amsart}

\usepackage{amsmath,amsfonts,amssymb,amsthm,epsfig,amscd}
\usepackage[]{fontenc}
\usepackage{xy}
\usepackage{enumerate}

 \allowdisplaybreaks \numberwithin{equation}{section}

\xyoption{all}

\def\step#1#2{{\emph{Step~#1}~[#2\,].}}

\numberwithin{equation}{section}

\newtheorem{theorem}{Theorem}[section]
\newtheorem{proposition}[theorem]{Proposition}

\newtheorem{lemma}[theorem]{Lemma}
\newtheorem{corollary}[theorem]{Corollary}

\theoremstyle{definition}
\newtheorem{definition}[theorem]{Definition}
\newtheorem{example}[theorem]{Example}
\newtheorem*{notation}{Notation}

\theoremstyle{remark}
\newtheorem{remark}[theorem]{Remark}

\newcommand{\gon}{\mathrm{gon}}
\newcommand{\degree}{\mathrm{d}}
\newcommand{\mult}{\mathrm{mult}}
\newcommand{\Supp}{\mathrm{Supp}}
\newcommand{\hl}{\mathrm{hl}}
\newcommand{\Alb}{\mathrm{Alb}}
\newcommand{\Pic}{\mathrm{Pic}}

\begin{document}

\title[On symmetric products of curves]{On symmetric products of curves}
\author{F. Bastianelli}
\address{Dipartimento di Matematica, Universit\`a degli Studi di Pavia, via Ferrata 1, 27100 Pavia}
\email{francesco.bastianelli@unipv.it}
\thanks{This work has been partially supported by 1) PRIN 2007 \emph{``Spazi di moduli e teorie di Lie''}; 2) INdAM (GNSAGA); 3) FAR 2008 (PV) \emph{``Variet\`a algebriche, calcolo algebrico, grafi orientati e topologici''.}}

\begin{abstract}
Let $C$ be a smooth complex projective curve of genus $g$ and let $C^{(2)}$ be its second symmetric product.
This paper concerns the study of some attempts at extending to $C^{(2)}$ the notion of gonality.
In particular, we prove that the degree of irrationality of $C^{(2)}$ is at least $g-1$ when $C$ is generic, and that the minimum gonality of curves through the generic point of $C^{(2)}$ equals the gonality of $C$.
In order to produce the main results we deal with correspondences on the $k$-fold symmetric product of $C$, with some interesting linear subspaces of $\mathbb{P}^n$ enjoying a condition of Cayley-Bacharach type, and with monodromy of rational maps.
As an application, we also give new bounds on the ample cone of $C^{(2)}$ when $C$ is a generic curve of genus ${6\leq g\leq 8}$.
\end{abstract}

\maketitle

\section{Introduction}

Let $C$ be a smooth irreducible complex projective curve of genus ${g\geq 0}$.
The \emph{gonality} of $C$ is the minimum positive integer $d$ such that $C$ admits a covering ${f\colon C\longrightarrow \mathbb{P}^1}$ of degree $d$ and we denote it by $\gon (C)$. The gonality is an important invariant of curves and it has been studied since the nineteenth century until now.
We deal throughout with two attempts at extending the notion of gonality to varieties of higher dimension.
In particular, we treat this topic on the second symmetric product $C^{(2)}$ of the curve $C$, which is the smooth surface parametrizing the unordered pairs of points of the curve.
We would like to point out that symmetric products of curves are very concrete projective varieties that are naturally defined by $C$.
Moreover, they somehow reflect the geometry of the curve and they are deeply involved in the classical theory of curves.
So it seems natural and interesting to study the problem of generalizing the notion of gonality on such varieties.

It is worth noticing that to compute the gonality of curves is a quite difficult task. Indeed, beside examples of morphisms reaching the expected minimum degree, one has to provide non-existence results for lower degrees. As it shall be clear in the following, similar remarks shall hold for both the notions we are going to introduce.

\smallskip
The most natural extension of gonality is probably the degree of irrationality.
Initially, it has been introduced in an algebraic context by Moh and Heinzer in \cite{MH}, whereas its geometric interpretation has been deeply studied
by Yoshihara \cite{TY, Y1, Y2}.
Given an irreducible complex projective variety $X$ of dimension $n$, the \emph{degree of irrationality} of $X$ is defined to be the integer
\begin{displaymath}
\degree_r(X):=\min\left\{d\in \mathbb{N}\,\left|\,\begin{array}{l} \textrm{there exists a dominant rational }\\  \textrm{map } F\colon X\dashrightarrow
\mathbb{P}^n \textrm{ of degree } d\end{array}\right. \right\}\,.
\end{displaymath}

Clearly, such a number is a birational invariant, and having ${\degree_r(X)=1}$ is equivalent to rationality.
Moreover, since any dominant rational map from a curve to $\mathbb{P}^1$ can be resolved to a morphism, the notion of degree of irrationality does
provide an extension of gonality to $n$-dimensional varieties.

We would like to recall that any dominant rational map ${C\dashrightarrow C'}$ between curves leads to the inequality ${\gon(C)\geq \gon(C')}$.
On the other hand, the existence of a dominant rational map ${X\dashrightarrow Y}$ between varieties of dimension ${n\geq 2}$ does not work analogously
on the degrees of irrationality.
Indeed there are counterexamples in the case of surfaces (cf. \cite{Y2, C}) and there are examples of non-rational threefolds that are unirational (see
for instance \cite{CG, IM}).

\smallskip
Turning to consider the second symmetric product $C^{(2)}$ of a smooth complex projective curve $C$ of genus $g$, we deal with the problem of computing its degree of irrationality.
Clearly, there is a strong connection between the existence of a dominant rational map ${F\colon C^{(2)}\dashrightarrow \mathbb{P}^2}$ and the genus of
the curve $C$.
For instance, rational and elliptic curves are such that the degree of irrationality of their second symmetric product is one and two respectively,
whereas we shall see that ${\degree_r(C^{(2)})\geq 3}$ for any curve of genus ${g\geq 2}$.

Furthermore, the degree of irrationality of the second symmetric product seems to depend on the existence of linear series on the curve as well.
Indeed, by using $g^r_d$'s on $C$ it is possible to construct rational dominant maps ${F\colon C^{(2)}\dashrightarrow \mathbb{P}^2}$ leading to the
following upper bound.
\begin{proposition}\label{proposition INTRO UPPER BOUND ON DEG IRR C(2)}
Let $C$ be a smooth complex projective curve. Let $\delta_1$ be the gonality of $C$ and for\, ${m=2,3}$, let $\delta_m$ be the minimum of the integers
$d$ such that $C$ admits a birational mapping onto a non-degenerate curve of degree $d$ in $\mathbb{P}^{m}$. Then
\begin{displaymath}
\degree_r(C^{(2)})\leq \min \left\{\delta_1^2\,,\, \frac{\delta_2(\delta_2-1)}{2}\,,\, \frac{(\delta_3-1)(\delta_3-2)}{2}-g\right\}\,.
\end{displaymath}
\end{proposition}

In the case of hyperelliptic curves of high genera the latter bound turns out to be an equality.
Namely
\begin{theorem}\label{theorem INTRO DEGIRR HYPERELLIPTIC}
Let $C$ be a smooth complex projective curve of genus $g\geq 2$ and assume that $C$ is hyperelliptic.
Then
\begin{itemize}
\item[$\mathrm{(i)}$] $3\leq \degree_r(C^{(2)})\leq 4$ when either $g=2$ or $g=3$;
\item[$\mathrm{(ii)}$] $\degree_r(C^{(2)})= 4$\, for any $g\geq 4$.
\end{itemize}
\end{theorem}
On the other hand, when the curve is assumed to be non-hyperelliptic, the situation is more subtle and it is no longer true that the degree of
irrationality of $C^{(2)}$ equals the square of the gonality of $C$ for high enough genus. The main result we prove on generic curves is the
following.
\begin{theorem}\label{theorem INTRO DEGIRR VERY GENERAL CURVES}
Let $C$ be a smooth complex projective curve of genus $g\geq 4$ and assume that $C$ is very general in the moduli space $\mathcal{M}_g$. Then
${\degree_r(C^{(2)})\geq g-1}$.
\end{theorem}
Without any assumption of generality on $C$, the latter inequality does not hold, but it is still possible to provide some estimations on the degree of
irrationality of second symmetric products of non-hyperelliptic curves. The following result summarizes the lower bounds we achieve and we list them by
genus.
\begin{theorem}\label{theorem INTRO DEGIRR NON-HYPERELLIPTIC}
Let $C$ be a smooth complex projective curve of genus $g\geq 3$ and assume that $C$ is non-hyperelliptic.
Then the following hold:
\begin{itemize}
\item[$\mathrm{(i)}$] if $g= 3,4$, then $\degree_r(C^{(2)})\geq 3$;
\item[$\mathrm{(ii)}$] if $g= 5$, then $\degree_r(C^{(2)})\geq 4$;
\item[$\mathrm{(iii)}$] if $g= 6$, then $\degree_r(C^{(2)})\geq 5$;
\item[$\mathrm{(iv)}$] if $g\geq 7$, then $\degree_r(C^{(2)})\geq \max \left\{\,6,\,\gon(C)\,\right\}$.
\end{itemize}
\end{theorem}

\smallskip
Another attempt to extend the notion of gonality to $n$-dimensional varieties is the following.
Given an irreducible complex projective variety $X$, we define the number
\begin{displaymath}
\degree_o(X):=\min\left\{d\in \mathbb{N}\,\left|\,\begin{array}{l} \textrm{there exists a family } \mathcal{E}=\left\{E_t\right\}_{t\in T} \\
\textrm{covering }X \textrm{ whose generic member is} \\ \textrm{an irreducible $d$-gonal curve} \end{array}\right. \right\}
\end{displaymath}
and we may call it the \emph{degree of gonality} of $X$.
Hence $\degree_o(X)$ is the minimum gonality of curves passing through the generic point of $X$.
Notice that the generic member $E_t$ is a possibly singular $d$-gonal curve, i.e. its normalization $\widetilde{E}_t$ admits a degree $d$ covering
${f_t\colon \widetilde{E}_t\longrightarrow\mathbb{P}^1}$.
The degree of gonality is a birational invariant, and $\degree_o(X)=1$ if and only if $X$ is an uniruled variety.
Moreover, ${\degree_o(C)=\gon(C)}$ for any complex projective curve $C$.

Although this second extension of the notion of gonality appears less intuitive and more artificial than the degree of irrationality, the degree of
gonality has a nice behavior with respect to dominance.
Namely, if there exists a dominant rational map ${X\dashrightarrow Y}$ between two irreducible complex projective varieties of dimension $n$, then
${\degree_o(X)\geq \degree_o(Y)}$ as in the one-dimensional case.

\smallskip
Dealing with the problem of computing the degree of gonality of the second symmetric product $C^{(2)}$ of a smooth complex projective curve $C$, it is
easy to check that ${\degree_o(C^{(2)})=1}$ when the curve is either rational or elliptic, and ${\degree_o(C^{(2)})=2}$ for any curve of genus two.
Moreover, we prove the following.
\begin{theorem}\label{theorem INTRO GONALITY OF MOVING CURVES}
Let $C$ be a smooth complex projective curve of genus ${g\geq 3}$.
For a positive integer $d$, let ${\mathcal{E}=\{E_t\}_{t\in T}}$ be a family of curves on $C^{(2)}$ parametrized over a smooth variety $T$, such that
the generic fiber $E_t$ is an irreducible $d$-gonal curve and for any point ${P\in C^{(2)}}$ there exists ${t\in T}$ such that ${P\in E_t}$.
Then ${d\geq \gon(C)}$.

\noindent Moreover, under the further assumption ${g\geq 6}$ and ${Aut(C)=\left\{Id_C\right\}}$, we have that equality holds if and only if $E_t$ is
isomorphic to $C$.
\end{theorem}
In particular, we have ${\degree_o(C^{(2)})\geq\gon(C)}$. Furthermore, for any smooth curve $C$, its second symmetric product is covered a family of
copies of $C$.
Hence ${\degree_o(C^{(2)})\leq \gon(C)}$ and the problem of computing the degree of gonality of second symmetric products of curves is now totally
understood.
\begin{theorem}\label{theorem INTRO DEGREE OF GONALITY}
Let $C$ be a smooth complex projective curve of genus ${g\geq 3}$. Then ${\degree_o(C^{(2)})=\gon(C)}$.
\end{theorem}

\smallskip
At the end of this paper, we present further an application of Theorem \ref{theorem INTRO GONALITY OF MOVING CURVES} improving the bounds of \cite{B}
on the nef cone of the second symmetric product of a generic curve $C$ of genus $g$. We would like to recall that the problem of describing the nef
cone $Nef(C^{(2)})_{\mathbb{R}}$ in the N\'eron-Severi space $N^1(C^{(2)})_{\mathbb{R}}$ is reduced to estimate the slope $\tau(C)$ of one of the rays
bounding the two-dimensional convex cone $Nef(C^{(2)})_{\mathbb{R}}$.

In \cite[Theorem 1]{B} are provided some bounds on $\tau(C)$ when the genus of $C$ is ${5\leq g\leq 8}$. The proof of such a result follows the
argument of \cite[Section 4]{R} -~which is based both on the main theorem of the latter paper and on the techniques introduced in \cite{EL}~- and
involves the gonality of moving curves on the second symmetric products.

By following the very same argument and by applying Theorem \ref{theorem INTRO GONALITY OF MOVING CURVES} to this setting, we achieve new bounds on the
nef cone of $C^{(2)}$ when $C$ has genus $6\leq g\leq 8$. Namely,
\begin{theorem}\label{theorem INTRO NEW BOUNDS}
Consider the rational numbers
\begin{equation*}
\tau_6= \frac{32}{13}\,,\quad \tau_7=\frac{77}{29}\quad\textrm{and}\quad \tau_8=\frac{17}{6}\,.
\end{equation*}
Let $C$ be a smooth complex projective curve of genus $6\leq g \leq 8$ and assume that $C$ is very general in the moduli space $\mathcal{M}_g$. Then
$\tau(C)\leq \tau_g$.
\end{theorem}

\smallskip
In order to prove the most of our results, the main technique is to use holomorphic differentials, following Mumford's method of induced differentials
(cf. \cite[Section 2]{Mu}).
In the spirit of \cite{LP}, we rephrase our settings in terms of correspondences on the product ${Y\times C^{(2)}}$, where $Y$ is an appropriate ruled
surface.
A general $0$-cycle of such a correspondence ${\Gamma\subset Y\times C^{(2)}}$ is a Cayley-Bacharach scheme with respect to the canonical linear series $|K_{C^{(2)}}|$, that is, any holomorphic 2-form vanishing on all but one the points of the $0$-cycle vanishes in the remaining point as well.
The latter property imposes strong conditions on the correspondence $\Gamma$, and the crucial point is to study the restrictions descending to the
second symmetric product and then to the curve $C$.

Another important technique involved in the proofs is monodromy.
In particular, we consider the generically finite dominant map ${\pi_1\colon \Gamma\longrightarrow Y}$ projecting a correspondence $\Gamma$ on the
first factor, and we study the action of the monodromy group of $\pi_1$ on the generic fiber. Finally, an important role is played by Abel's theorem
and some basic facts of Brill-Noether theory.

\smallskip
The plan of the paper is the following. Section \ref{section PRELIMINARIES} concerns preliminaries on symmetric products of curves and monodromy, whereas in Section \ref{section CORRESPONDENCES} we develop the main techniques to menage our problems.
In particular, given a smooth curve $C$ of genus $g$ and its $k$-fold symmetric product, with ${2\leq k\leq g-1}$, we investigate how the existence of a correspondence on $C^{(k)}$ influences the geometry of the curve itself (see Theorem \ref{theorem CORRESPONDENCES ON C(k)} and Corollary \ref{corollary CORRESPONDENCES ON C(k)}).

Dealing with this issue, we come across linear subspaces of $\mathbb{P}^n$ satisfying a condition of Cayley-Bacharach type.
We spend the whole Section \ref{section LINEAR SUBSPACES} to analyze them.
They turn out to enjoy interesting properties both on the dimension of their linear span, and on their configuration in the projective space (cf. Theorem \ref{theorem CB FOR LINEAR SUBSPACES bis}).

Section \ref{section DEGREE OF GONALITY} and Section \ref{section DEGREE OF IRRATIONALITY} are devoted to study the degree of gonality and the degree
of irrationality respectively.

Finally, in the last section we deal with the ample cone on second symmetric products of curves and we prove Theorem \ref{theorem INTRO NEW BOUNDS}.

\smallskip
\begin{notation}
We shall work throughout over the field $\mathbb{C}$ of complex numbers. Given a variety $X$, we say that a property holds for a \emph{general} point
${x\in X}$ if it holds on an open non-empty subset of $X$. Moreover, we say that ${x\in X}$ is a \emph{very general} -~or \emph{generic}~- point if
there exists a countable collection of proper subvarieties of $X$ such that $x$ is not contained in the union of those subvarieties.
By \emph{curve} we mean a complete reduced algebraic curve over the field of complex numbers. When we speak of \emph{smooth} curve, we always
implicitly assume it to be irreducible.
\end{notation}

\bigskip
\section{Preliminaries}\label{section PRELIMINARIES}

\subsection{Definition and first properties}

Let $C$ be a smooth complex projective curve of genus ${g\geq 0}$.
For an integer ${k\geq 1}$, let ${C^k:=C\times\ldots\times C}$ denote its $k$-fold ordinary product and let $S_k$ be the $k$-th symmetric group.
We define the $k$\emph{-fold symmetric product} of $C$ as the quotient
\begin{equation*}
C^{(k)}:=\frac{C^k}{S_k}
\end{equation*}
under the action of $S_k$ permuting the factors of $C^k$.
Hence the quotient map ${\pi\colon C^k\longrightarrow C^{(k)}}$ sending ${(p_1,\ldots,p_k)\in C^k}$ to the point ${p_1+\ldots+p_k\in C^{(k)}}$ has
degree $k!$.
The $k$-fold symmetric product is a smooth projective variety of dimension $k$ (cf. \cite[p. 18]{ACGH}) and it parametrizes the effective divisors on
$C$ of degree $k$ or, equivalently, the unordered $k$-tuples of points of $C$ .

\subsection{Linear series and subordinate loci}

Let $d$ and $r$ be some positive integers. As customary, we denote by ${W^r_d(C)\subset \Pic^d(C)}$ the subvariety parametrizing the complete linear
systems on $C$ of degree $d$ and dimension at least $r$. We recall that the dimension of $W^r_d(C)$ is bounded from below by the Brill-Noether number
${\rho(g,r,d):=g-(r+1)(g-d+r)}$, and if the curve $C$ is very general in the moduli space $\mathcal{M}_g$, then $\dim W^r_d(C)$ equals $\rho(g,r,d)$.

Let $G^r_d(C)$ be the variety of linear series on $C$ of degree $d$ and dimension exactly $r$, whose points are said $g^r_d$'s. We note that the
gonality of $C$ is the minimum $d$ such that $C$ admits a $g^1_d$. Moreover, any complete $g^r_d$ on $C$ can be thought as an element of $W^r_d(C)$.

Given a linear series $\mathcal{D}\in G^r_d(C)$, we define the locus of divisors on $C$ subordinate to $\mathcal{D}$ as
\begin{equation}\label{equation SUBORDINATE LOCUS}
\Gamma_k\left(\mathcal{D}\right):=\left\{P\in C^{(k)} \left| D-P\geq 0 \textrm{ for some }D\in\mathcal{D}\right.\right\}.
\end{equation}
We point out that the linear series $\mathcal{D}$ is not assumed to be base-point-free. Furthermore, the locus $\Gamma_k\left(\mathcal{D}\right)$ is a
subvariety of $C^{(k)}$ and if the dimension of $\mathcal{D}$ is ${r=k-1}$, then $\Gamma_k\left(\mathcal{D}\right)$ is a divisor.

\subsection{Canonical divisor on $C^{(k)}$}

Let ${\phi\colon C\longrightarrow \mathbb{P}^{g-1}}$ be the canonical map of the smooth curve $C$ of genus $g$. For ${1\leq k\leq g-1}$, let us
consider the $k$-fold symmetric product $C^{(k)}$ and the Grassmannian variety $\mathbb{G}(k-1,g-1)$ parametrizing ${(k-1)}$-dimensional planes in
$\mathbb{P}^{g-1}$. As $\phi(C)$ is a non-degenerate curve of $\mathbb{P}^{g-1}$, by General Position Theorem it is well defined the \emph{Gauss map}
\begin{equation}\label{equation GAUSS MAP}
\mathcal{G}_k\colon C^{(k)}\dashrightarrow \mathbb{G}(k-1,g-1)
\end{equation}
sending a point ${p_1+\ldots+p_k\in C^{(k)}}$ to the linear span of the $\phi(p_i)$'s in $\mathbb{P}^{g-1}$.

Let $|K_{C^{(k)}}|$ be the canonical linear system on $C^{(k)}$ and let ${\psi_k\colon C^{(k)}\longrightarrow \mathbb{P}\left(H^{k,0}(C^{(k)})\right)}$
be the induced canonical map. Since
\begin{equation}\label{equation CANONICAL LINEAR SERIES ON C(k)}
H^{k,0}\left(C^{(k)}\right)\cong \bigwedge^k H^{1,0}(C)
\end{equation}
(cf. \cite{Ma}), we have the following commutative diagram
\begin{displaymath}
\xymatrix{ C^{(k)} \ar@{-->}[dr]_{\mathcal{G}_k} \ar@{-->}[rr]^{\psi_k} & & \mathbb{P}^{N}  \\ & \mathbb{G}(k-1,g-1) \ar[ur]_{p} & \\ }
\end{displaymath}
where ${N:={g \choose k}-1}$ and ${p\colon \mathbb{G}(k-1,g-1)\longrightarrow \mathbb{P}^{N}}$ is the Pl\"ucker embedding.

We recall that for any ${L\in \mathbb{G}(g-k-1,g-1)}$, the Schubert cycle ${\sigma_1(L):=\left\{l\in \mathbb{G}(k-1,g-1)\mid l\cap L\neq \emptyset
\right\}}$ maps into a hyperplane section of ${p\left(\mathbb{G}(k-1,g-1)\right)\subset \mathbb{P}^{N}}$. Then it is possible to provide canonical
divisors on the $k$-fold symmetric products of $C$ as follows.
\begin{lemma}\label{lemma CANONICAL DIVISORS ON C(k)}
For any ${L\in \mathbb{G}(g-k-1,g-1)}$, let ${\pi_L\colon \phi(C)\dashrightarrow \mathbb{P}^{k-1}}$ be the projection from the $(g-k-1)$-plane $L$ of
the canonical image of $C$ and let $\mathcal{D}_L$ be the associated linear series on $C$ -~not necessarily base-point-free~- of degree $2g-2$ and
dimension $k-1$. Then the effective divisor $\Gamma_k(\mathcal{D}_L)$ defined in (\ref{equation SUBORDINATE LOCUS}) is a canonical divisor of
$C^{(k)}$, that is $\Gamma_k(\mathcal{D}_L)\in |K_{C^{(k)}}|$.
\end{lemma}

In particular, a generic point ${P=p_1+\ldots+p_k\in C^{(k)}}$ lies on the divisor $\Gamma_k(\mathcal{D}_L)$ if and only if the linear span of the
$\phi(p_j)$'s in $\mathbb{P}^{g-1}$ is a point of the Schubert cycle $\sigma_1(L)$, that is $\mathcal{G}_k(P)$ intersects $L$.

\subsection{Monodromy}

To conclude this section, we follow \cite{H} to recall some basic facts on the monodromy of a generically finite dominant morphism ${F\colon
X\longrightarrow Y}$ of degree $d$ between irreducible complex algebraic varieties of the same dimension.

Let ${U\subset Y}$ be a suitable Zariski open subset of $X$ such that the restriction ${F^{-1}(U)\longrightarrow U}$ is an unbranched covering of
degree $d$.
Given a generic point ${y\in U}$, by lifting loops at $y$ to $F^{-1}(U)$, we may define the monodromy representation ${\rho\colon \pi_1(U,y)
\longrightarrow Aut\left(F^{-1}(y)\right)\cong S_d}$ and we define the \emph{monodromy group} $M(F)$ of $F$ to be the image of the latter
homomorphism.

Equivalently, let $L$ be the normalization of the algebraic field extension $K(X)/K(Y)$ of degree $d$, and let $Gal(L/K(Y))$ be the \emph{Galois group}
of $L/K(Y)$, that is the group of the automorphisms of the field $L$ fixing every element of $K(Y)$. Then the monodromy group $M(F)$ and the Galois
group $Gal(L/K(Y))$ are isomorphic (see \cite[p. 689]{H}). In particular, this implies that the monodromy group of $F$ is independent of the choice of
the Zariski open set $U$. Moreover, $F$ should not be necessarily a morphism, but it suffices being a dominant rational map.

The simple fact we want to point out is that the action of $M(F)$ on the fiber $F^{-1}(y)$ is transitive, because of the connectedness of $X$.
Roughly speaking, this means that the points of the fiber over a generic point are undistinguishable.
Namely, suppose that a point $x_i\in F^{-1}(y)$ enjoys some special property such that as we vary continuously the point $y$ on a suitable open subset
$U\subset Y$, that special property is preserved as we follow the correspondent point of the fiber.
Then for any loop $\gamma\in \pi_1(U,y)$, we have that the ending point $x_j\in F^{-1}(y)$ of the unique lifting $\widetilde{\gamma}$ of $\gamma$
starting from $x_i$ must enjoy the same property. Hence the transitivity of the action assures that there is no way to distinguish a point of the fiber
over $y\in Y$ from another for enjoying a property as above.

\bigskip
\section{Linear subspaces of $\mathbb{P}^n$ in special position}\label{section LINEAR SUBSPACES}

In this section we deal with sets of linear subspaces of the $n$-dimensional projective space satisfying a condition of Cayley-Bacharach type. In
particular, we shall provide a bound on the dimension of their linear span in $\mathbb{P}^n$, and we shall present some examples proving the sharpness of the bound. The reasons for studying these particular linear spaces shall be clear in the next section, when we shall relate them to correspondences on symmetric products of curves.

\smallskip
To start we recall the following definition (cf. \cite{GH1}).
\begin{definition}\label{definition CAYLEY-BACHARACH CONDITION}
Let $\mathcal{D}$ be a complete linear system on a projective variety $X$. We say that a $0$-cycle ${P_1+\ldots+P_d\subset X^{(d)}}$ \emph{satisfies the Cayley-Bacharach condition with respect to} $\mathcal{D}$ if for every ${i=1,\ldots,d}$ and for any effective divisor ${D\in \mathcal{D}}$ passing through ${P_1,\ldots,\widehat{P}_i,\ldots,P_d}$, we have $P_i\in D$ as well.
\end{definition}

\smallskip
Let $n$ and $k$ be two integers with ${n\geq k \geq 2}$, and let $\mathbb{G}(k-1,n)$ be the Grassmann variety of $(k-1)$-planes in $\mathbb{P}^n$. For
an integer ${d\geq 2}$, let us consider a set ${\left\{l_1,\ldots,l_d\right\}\subset \mathbb{G}(k-1,n)}$ and suppose that the associated $0$-cycle
${l_1+\ldots+l_d}$ satisfies the Cayley-Bacharach condition with respect to the complete linear series $|\mathcal{O}_{\mathbb{G}(k-1,n)}(1)|$.

We recall that for any ${L\in \mathbb{G}(n-k,n)}$, the Schubert cycle ${\sigma_1(L):=\{l\in\mathbb{G}(k-1,n)|l\cap L\neq\emptyset\}}$ is an effective
divisor of $|\mathcal{O}_{\mathbb{G}(k-1,n)}(1)|$. Thus the set ${\left\{l_1,\ldots,l_d\right\}}$ is such that for every ${i=1,\ldots,d}$ and for any
${L\in\mathbb{G}(n-k,n)}$ with ${l_1,\ldots,\widehat{l}_i,\ldots,l_d\in \sigma_1(L)}$, we have ${l_i\in \sigma_1(L)}$ as well. Then it makes sense to
give the following definition expressing a condition of Cayley-Bacharach type for $(k-1)$-dimensional linear subspaces of $\mathbb{P}^{n}$.
\begin{definition}\label{definition CB FOR LINEAR SUBSPACES}
We say that the $(k-1)$-planes ${l_1,\ldots,l_d\subset \mathbb{P}^{n}}$ are in \emph{special position with respect to $(n-k)$-planes} if for every
${i=1,\ldots,d}$ and for any $(n-k)$-plane ${L\subset \mathbb{P}^n}$ intersecting ${l_1,\ldots,\widehat{l}_i,\ldots,l_d}$, we have ${l_i\cap L\neq
\emptyset}$.
\end{definition}
We note that the $(k-1)$-planes in the definition are not assumed to be distinct. In particular, it is immediate to check that two $(k-1)$-planes
${l_1,l_2\subset \mathbb{P}^n}$ are in special position if and only if they coincide.

The main result of this section is the following.
\begin{theorem}\label{theorem CB FOR LINEAR SUBSPACES bis}
Let ${2\leq k,d \leq n}$ be some integers and suppose that the $(k-1)$-planes ${l_1,\ldots,l_d\subset \mathbb{P}^n}$ are in special position with
respect to $(n-k)$-planes of $\mathbb{P}^n$. Then the dimension of their linear span ${S:=Span(l_1,\ldots,l_d)}$ in $\mathbb{P}^n$ is ${s\leq
\left[\frac{kd}{2}\right]-1}$.
\end{theorem}

In order to prove this result, let us state the following preliminary lemma.

\begin{lemma}\label{lemma ALL BUT ONE}
Under the assumption of Theorem \ref{theorem CB FOR LINEAR SUBSPACES bis}, suppose further that there exists a linear space ${R\subset \mathbb{P}^n}$
containing ${l_1,\ldots,\widehat{l}_j,\ldots,l_d}$. Then ${l_j\subset R}$ as well.
\begin{proof}
Let $r$ denote the dimension of $R$. If $r=n$ the statement is trivially true, then let us assume ${r<n}$. As ${k-1\leq r}$ we have that ${0\leq
r-k+1\leq n-k}$ and we can consider a $(r-k+1)$-plane ${T\subset R}$. Then $T$ intersects each of the $(k-1)$-planes
${l_1,\ldots,\widehat{l}_j,\ldots,l_d}$. Therefore by special position property, any $(n-k)$-plane $L$ containing $T$ must intersect $l_j$, thus
${l_j\cap T\neq\emptyset}$. Therefore $l_j$ meets every $(r-k+1)$-plane ${T\subset R}$ and hence $l_j\subset R$.
\end{proof}
\end{lemma}

\begin{proof}[\textit{Proof of Theorem \ref{theorem CB FOR LINEAR SUBSPACES bis}}]
Let us fix ${2\leq k,d\leq n}$. Notice that if ${n\leq \left[\frac{kd}{2}\right]-1}$, the statement is trivially proved. Hence we assume hereafter
${n\geq \left[\frac{kd}{2}\right]}$. We proceed by induction on the number $d$ of $(k-1)$-planes.

\smallskip
Let ${l_1,l_2\subset \mathbb{P}^n}$ be two $(k-1)$-dimensional planes in special position with respect to $(n-k)$-planes. Then we set ${R:=l_1}$ and
Lemma \ref{lemma ALL BUT ONE} implies ${l_2\subset R}$. Hence ${R=l_1=l_2}$ and ${[\frac{kd}{2}]-1=k-1=\dim R}$. Thus the statement is proved when
${d=2}$.

\smallskip
By induction, suppose that the assertion holds for any ${2\leq h\leq d-1}$ and for any $h$-tuple of $(k-1)$-dimensional linear subspaces of
$\mathbb{P}^m$ in special position with respect to $(m-k)$-planes, with ${m\geq h}$.

Now, let ${l_1,\ldots,l_d\subset \mathbb{P}^n}$ be $(k-1)$-planes in special position with respect to $(n-k)$-planes.

We first consider the case where is not possible to choose one of the $l_i$'s such that it does not coincide with any of the others. In this situation,
the number of distinct $l_i$'s is at most ${\left[\frac{d}{2}\right]}$. Thus the dimension of their linear span in $\mathbb{P}^n$ is at most
${k\left[\frac{d}{2}\right]-1\leq \left[\frac{kd}{2}\right]-1}$ as claimed.

Then we consider the $(k-1)$-plane $l_1$ and we suppose - without loss of generality - that it does not coincide with any of the others $l_i$'s.
Therefore is possible to choose a point ${p\in l_1}$ such that ${p\not\in l_i}$ for any ${i=2,\ldots,d}$. Moreover, let ${H\subset \mathbb{P}^n}$ be an
hyperplane not containing $p$ and consider the projection
\begin{displaymath}
\begin{array}{cccl}
\pi_{p}\colon & \mathbb{P}^n-\{p\} & \longrightarrow & H \cong \mathbb{P}^{n-1}\\
 & q & \longmapsto & \overline{pq}\cap H\,.
\end{array}
\end{displaymath}

For ${2\leq i\leq d}$, let ${\lambda_i:=\pi_p\left(l_i\right)\subset H}$ be the image of $l_i$ on $H$. We claim that the $(k-1)$-planes
${\lambda_2,\ldots,\lambda_d\subset H}$ are in special position with respect to $(n-1-k)$-planes of ${H\cong \mathbb{P}^{n-1}}$. To see this fact, let
${j\in\{2,\ldots,d\}}$ and let ${\Lambda\subset H}$ be a $(n-1-k)$ plane intersecting ${\lambda_2,\ldots,\widehat{\lambda}_j,\ldots,\lambda_d}$. Since
${p\in l_1}$, it follows that the $(n-k)$-plane ${L:=Span(\Lambda,p)\subset \mathbb{P}^n}$ intersects ${l_1,\ldots,\widehat{l}_j,\ldots,l_d}$. As they
are in special position with respect to $(n-k)$-planes, we have that $L$ intersects $l_j$ as well. Then, given a point ${q_j\in L\cap l_j}$, we have
that ${\pi_p(q_j)\in \Lambda}$. In particular, $\Lambda$ meets $\lambda_j$ at $\pi_p(q_j)$ and hence ${\lambda_2,\ldots,\lambda_d\subset H}$ are in
special position with respect to $(n-1-k)$-planes of the hyperplane ${H\cong \mathbb{P}^{n-1}}$.

By induction, the linear span ${\Sigma:=Span(\lambda_2,\ldots,\lambda_d)\subset H}$ has dimension ${\dim \Sigma\leq \left[\frac{k(d-1)}{2}\right]-1}$.
Then the linear space ${R:=Span (\lambda_2,\ldots,\lambda_d,p)\subset \mathbb{P}^n}$ has dimension ${\dim R=\dim\Sigma+1\leq
\left[\frac{k(d-1)}{2}\right]\leq \Big[\frac{kd}{2}\Big]-1}$ for any ${k\geq 2}$. Notice that $R$ contains ${l_2,\ldots,l_d}$. Hence ${l_1\subset R}$
as well by Lemma \ref{lemma ALL BUT ONE}. Thus $R$ contains the linear span in $\mathbb{P}^n$ of all the $l_i$'s and the assertion follows.
\end{proof}

We would like to note that the assumption ${k\geq 2}$ in Theorem \ref{theorem CB FOR LINEAR SUBSPACES bis} is necessary. For instance, let ${k=1}$ and
consider three collinear points in $\mathbb{P}^n$. Clearly, they are in special position with respect to $(n-1)$-planes and ${[\frac{kd}{2}]-1=0}$, but
they span a line.

On the other hand, when ${k=2}$ the theorem assures that if ${\{l_1,\ldots,l_d\}}$ is a set of $d$ lines in special position with respect to
$(n-2)$-planes of $\mathbb{P}^n$, then their linear span has dimension lower than $d$. The following examples concern the configuration of such lines
in $\mathbb{P}^{d-1}$ and show that the latter bound is sharp.

\begin{example}\label{example THREE LINES IN Pn}
Let us consider three distinct lines ${l_1,l_2,l_3}$ in $\mathbb{P}^n$. Then they are in special position with respect to $(n-2)$-planes if and only if
they lie on a plane ${\pi\subset \mathbb{P}^n}$ and they meet at a point ${p\in\pi}$.

Indeed, suppose that ${l_1,l_2,l_3}$ are in special position with respect to $(n-2)$-planes. Therefore they must lie on a plane ${\pi\subset
\mathbb{P}^{n}}$ by Theorem \ref{theorem CB FOR LINEAR SUBSPACES bis}. Then consider the point ${p=l_2\cap l_3}$ and let $L$ be a $(n-2)$-plane such
that ${L\cap \pi=\{p\}}$. Thus $L$ must intersect also $l_1$ by special position property, and hence ${p\in l_1}$.

On the other hand, it is immediate to check that if three distinct lines of $\mathbb{P}^n$ lie on the same plane and meet at a point, then any
$(n-2)$-plane intersecting two of them intersects the last one as well.
\end{example}

\begin{example}
Let ${l_1,\ldots,l_4\subset \mathbb{P}^n}$ be four skew lines. Then they are in special position with respect to $(n-2)$-planes if and only if they lie
on the same ruling of a quadric surface ${Q\subset \mathbb{P}^3}$. In particular, the $l_i$'s span the whole $\mathbb{P}^3$.

If the $l_i$'s enjoy special position property, they span a linear subspace ${S\cong \mathbb{P}^3}$ by Theorem \ref{theorem CB FOR LINEAR SUBSPACES
bis}. Let ${Q\subset S}$ be the quadric defined as the union of the lines intersecting $l_1$, $l_2$ and $l_3$. By special position property, any line
${L\subset S}$ intersecting ${l_1,l_2,l_3}$ must meet $l_4$ too. Hence ${l_4\subset Q}$ as well and it lies on the same ruling of the other $l_i$'s.

To see the converse, it suffices to observe that any quadric surface ${Q\subset \mathbb{P}^3}$ is covered by two families of skew lines, $\mathcal{L}$
and $\mathcal{L}'$, such that any two lines ${l\in \mathcal{L}}$ and ${l'\in \mathcal{L}'}$ meet at a point (see e.g. \cite[p. 478]{GH}).
\end{example}

\begin{example}
In general, if ${l_1,\ldots,l_d\subset \mathbb{P}^{d-1}}$ are skew lines lying on a non-degenerate surface ${Q\subset \mathbb{P}^{d-1}}$ of minimal
degree, then they are in special position with respect to $(d-3)$-planes.

Under these assumptions, $Q$ is a ruled surface of degree ${d-2}$ (cf \cite[p. 522]{GH}). If ${L\subset \mathbb{P}^{d-1}}$ is a $(d-3)$-plane
intersecting ${l_1,\ldots,l_{d-1}}$, then ${L\cap \Sigma}$ is a curve $C$ of degree ${\deg C\leq d-2}$. In particular, $C$ does not lies on the ruling
of $Q$ and hence it must intersect $l_d$ too.
\end{example}

\bigskip
\section{Correspondences with null trace on symmetric products of curves}\label{section CORRESPONDENCES}

In order to deal with correspondences with null trace on symmetric products of curves, we would like to recall the basic properties of Mumford's
induced differentials (see \cite[Section 2]{Mu}) and their applications to the study of correspondences (cf. \cite[Section 2]{LP}).
Then we shall turn to symmetric products of curves and we shall prove the main result of this section, which gives a geometric interpretation of the
existence correspondences with null trace on these varieties.

\smallskip
Let $X$ and $Y$ be two projective varieties of dimension $n$, with $X$ smooth and $Y$ integral.
\begin{definition}\label{definition CORRESPONDENCE}
A \emph{correspondence of degree} $d$ \emph{on} $Y\times X$ is a reduced $n$-dimensional variety ${\Gamma\subset Y\times X}$ such that the projections
${\pi_1\colon \Gamma\longrightarrow Y}$, ${\pi_2\colon \Gamma\longrightarrow X}$ are generically finite dominant morphisms and ${deg\,\pi_1=d}$.
Moreover, if ${deg\,\pi_2=d'}$ we say that $\Gamma$ is a $(d,d')$-\emph{correspondence}.
\end{definition}
So, let ${\Gamma\subset Y\times X}$ be a correspondence of degree $d$. Let ${X^d=X\times\ldots\times X}$ be the $d$-fold ordinary product of $X$ and
let ${p_i\colon X^d\longrightarrow X}$ be the $i$-th projection map, with ${i=1,\ldots,d}$. Let us consider the $d$-fold symmetric product
${X^{(d)}=X^d/S_d}$ of $X$ together with the quotient map ${\pi\colon X^d\longrightarrow X^{(d)}}$. Then we define the set ${U:=\{y\in
Y_{reg}\,|\,\dim\pi_1^{-1}(y)=0\}}$ and the morphism ${\gamma\colon U \longrightarrow X^{(d)}}$ given by ${\gamma(y):= P_1+\ldots +P_d}$, where
${\pi_1^{-1}(y)=\{(y,P_i)\,|\,i=1,\ldots,d\}}$.

By using Mumford's induced differentials, we want to define the trace map of $\gamma$. To this aim, we consider a holomorphic $n$-form ${\omega\in
H^{n,0}(X)}$ and the $(n,0)$-form
\begin{equation*}
\omega^{(d)}:=\sum_{i=1}^d p_i^*\omega \in H^{n,0}(X^{d}),
\end{equation*}
which is invariant under the action of $S_d$. Thus for any smooth variety $W$ and for any morphism ${f\colon W\longrightarrow X^{(d)}}$, there exists a
canonically induced $(n,0)$-form $\omega_{f}$ on $W$ (cf. \cite[Section 2]{Mu}). In particular, we define the Mumford's \emph{trace map} of $\gamma$
as
\begin{displaymath}
\begin{array}{ccccl}
Tr_{\gamma}\colon & H^{n,0}(X) & \longrightarrow & H^{n,0}(U) \\
 & \omega & \longmapsto & \omega_{\gamma}
\end{array}\, .
\end{displaymath}

Another way to define the trace map of $\gamma$ is the following. Let us consider the sets ${V:=\{y\in U\,|\pi_1^{-1}(y)\textrm{ has } d \textrm{
distinct points}\}}$ and
\begin{equation*}
X_0^{(d)}:=\pi\left(X^d-\bigcup_{i,j}\Delta_{i,j}\right),
\end{equation*}
where $\Delta_{i,j}$ is the $(i,j)$-diagonal of $X^{d}$, with ${i,j=1,\ldots, d}$ and ${i\neq j}$. Moreover, let us define the map
\begin{displaymath}
\begin{array}{cccc}
\delta_d\colon & H^{n,0}(X) & \longrightarrow & H^{n,0}(X_0^{(d)}) \\
 & \omega & \longmapsto & \pi_*(\omega^{(d)})
\end{array}\, ,
\end{displaymath}
i.e. $\omega^{(d)}$ is thought as a $(n,0)$-form on $X_0^{(d)}$. Then ${Im\,\gamma_{|V}\subset X_0^{(d)}}$ and the Mumford's trace map of $\gamma$
turns out to be ${Tr_{\gamma} = \gamma^*_{|V}\circ \delta_d}$ (cf. \cite[Proposition 2.1]{LP}).

\smallskip
The following result shows that the property of having null trace imposes strong conditions on the correspondence ${\Gamma\subset Y\times X}$. We would
like to note that in \cite[Proposition 2.2]{LP} it is presented in the case of correspondences on surfaces, but it is still true when $X$ and $Y$ are
$n$-dimensional varieties and the proof follows the very same argument.
\begin{proposition}\label{proposition LOPEZ PIROLA}
Let $X$ and $Y$ be two projective varieties of dimension $n$, with $X$ smooth and $Y$ integral. Let $\Gamma$ be a correspondence of degree $d$ on
${Y\times X}$ with null trace. Let ${y\in Y_{reg}}$ such that ${\dim\pi_1^{-1}(y)=0}$ and let ${\pi_1^{-1}(y)=\{(y,P_i)\in\Gamma\,|\,i=1,\ldots,d\}}$
be its fiber. Then the $0$-cycle ${P_1+\ldots+P_d}$ satisfies the Cayley-Bacharach condition with respect to the canonical linear series $|K_X|$, that
is for every ${i=1,\ldots,d}$ and for any effective canonical divisor $K_X$ containing ${P_1,\ldots,\widehat{P}_i,\ldots,P_d}$, we have ${P_i\in
K_X}$.
\end{proposition}

\smallskip
Now we turn to symmetric products of curves and we state the main result of this section. Let us consider a smooth projective curve $C$ of genus ${g}$
and let $C^{(k)}$ be its $k$-fold symmetric product, with ${2\leq k\leq g-1}$. Let us denote by ${\phi\colon C\longrightarrow \mathbb{P}^{g-1}}$ the
map induced by the canonical linear series $|K_C|$ on $C$. The following result connects the existence of a correspondence with null trace on $C^{(k)}$ and the geometry of the canonical image of $C$.

\begin{theorem}\label{theorem CORRESPONDENCES ON C(k)}
Let $C$ be a smooth projective curve of genus $g$ and let $Y$ be a projective integral variety of dimension ${2\leq k\leq g-1}$. Let $\Gamma$ be a
correspondence of degree $d\geq 2$ on ${Y\times C^{(k)}}$ with null trace. For a generic point ${y\in Y_{reg}}$, let
${\pi_1^{-1}(y)=\{(y,P_i)\in\Gamma\,|\,i=1,\ldots,d\}}$ be its fiber, where ${P_i=p_{i1}+\ldots+p_{ik}\in C^{(k)}}$ for ${i=1,\ldots,d}$.
Then the linear span of all the $\phi(p_{ij})$'s in $\mathbb{P}^{g-1}$ has dimension
\begin{displaymath}
s\leq \left[\frac{k\,d}{2}\right]-1.
\end{displaymath}
\begin{proof}
By Proposition \ref{proposition LOPEZ PIROLA} the $0$-cycle ${P_1+\ldots+P_d}$ satisfies the Cayley-Bacharach condition with respect to the canonical
linear series $|K_{C^{(k)}}|$. For any ${(g-k-1)}$-plane ${L\subset \mathbb{P}^{g-1}}$, Lemma \ref{lemma CANONICAL DIVISORS ON C(k)} assures that the
subordinate locus $\Gamma_k(\mathcal{D}_L)$ is an effective canonical divisor on $C^{(k)}$. Therefore for every ${i=1,\ldots,d}$ and for any ${L\in
\mathbb{G}(g-k-1,g-1)}$ with ${P_1,\ldots,\widehat{P_i},\ldots,P_d\in \Gamma_k(\mathcal{D}_L)}$, we have ${P_i\in \Gamma_k(\mathcal{D}_L)}$ as well.

Now, let ${\mathcal{G}_k\colon C^{(k)} \dashrightarrow \mathbb{G}(k-1,g-1)}$ be the Gauss map sending a point ${P=p_1+\ldots+p_k}$ to the linear span
of the $\phi(p_j)$'s in $\mathbb{P}^{g-1}$. We recall that for any ${L\in \mathbb{G}(g-k-1,g-1)}$ and for any ${P\in C^{(k)}}$, we have that ${P\in
\Gamma_k(\mathcal{D}_L)}$ if and only if $\mathcal{G}_k(P)$ intersects $L$. Thus the $0$-cycle ${P_1+\ldots+P_d}$ is such that for every
${i=1,\ldots,d}$ and for any ${L\in \mathbb{G}(g-k-1,g-1)}$ intersecting the $(k-1)$-planes
${\mathcal{G}_k(P_1),\ldots,\widehat{\mathcal{G}_k(P_i)},\ldots,\mathcal{G}_k(P_d)}$, we have ${\mathcal{G}_k(P_i)\cap L\neq \emptyset}$.

In particular, the $(k-1)$-planes ${\mathcal{G}_k(P_1),\ldots,\mathcal{G}_k(P_d)}$ are in special position with respect to $(g-k-1)$-planes (cf. Definition \ref{definition CB FOR LINEAR SUBSPACES}). Therefore Theorem \ref{theorem CB FOR LINEAR SUBSPACES bis} assures that their linear span in $\mathbb{P}^{g-1}$ has dimension lower than $\left[\frac{kd}{2}\right]$. Since any $\mathcal{G}_k(P_i)$ is generated by ${\phi(p_{i1}),\ldots,\phi(p_{ik})}$, we conclude that the linear span of all the $\phi(p_{ij})$'s in $\mathbb{P}^{g-1}$ has dimension bounded by ${\left[\frac{kd}{2}\right]-1}$ as claimed.
\end{proof}
\end{theorem}

As we anticipated, Theorem \ref{theorem CORRESPONDENCES ON C(k)} shall turn out to be very useful to deal with degree of irrationality and degree of
gonality on second symmetric products of curves. In particular, when ${k=2}$ the latter result assures the $\phi(p_{ij})$'s span a linear subspace of
$\mathbb{P}^{g-1}$ of dimension ${s\leq d-1}$.

\smallskip
The following assertion is an immediate consequence of Theorem \ref{theorem CORRESPONDENCES ON C(k)} connecting the existence of correspondences with
null trace on $C^{(k)}$ and the existence of complete linear series on $C$.
\begin{corollary}\label{corollary CORRESPONDENCES ON C(k)}
Under the assumption of Theorem \ref{theorem CORRESPONDENCES ON C(k)}, suppose in addiction that $C$ is non-hyperelliptic and that the number of
distinct $p_{ij}$'s is ${m>[\frac{kd}{2}]}$. Then $C$ possesses a complete $g_m^r$ with ${r\geq 1}$.
\proof
Let $m$ be the number of distinct $p_{ij}$'s on $C$ and let us denote by ${q_1,\ldots,q_m}$ these points. Consider the divisor ${D=q_1+\ldots+q_m}$ of
degree $m$ on $C$. As the curve $C$ is non-hyperelliptic, the canonical map $\phi$ is an embedding and the $\phi(q_t)$'s are all distinct. Hence their linear span
in $\mathbb{P}^{g-1}$ has dimension lower than $[\frac{kd}{2}]$. Therefore by the geometric version of Riemann-Roch theorem we have
\begin{equation*}
\dim |D|=m-1-\dim \overline{\phi(D)}\geq m-\left[\frac{kd}{2}\right]\geq 1.
\end{equation*}
Thus ${|D|=|q_1+\ldots+q_m|}$ is a complete $g^r_m$ on $C$ with ${r\geq 1}$.
\endproof
\end{corollary}

\begin{remark}
In \cite{GH1}, Griffiths and Harris study $0$-cycles on an algebraic variety $X$ satisfying the Cayley-Bacharach condition with respect to a complete
linear system $|D|$. In particular, given such a $0$-cycle ${P_1+\ldots+P_d}$ and the rational map ${\phi_{|D|}\colon X\dashrightarrow \mathbb{P}^r}$,
they present some results on the dimension of the linear span of the $\phi_{|D|}(P_i)$'s in $\mathbb{P}^r$ and, consequently, on the existence of linear
series on $X$. We note that we start from an analogous situation with ${X=C^{(k)}}$, but the results of this section deal with the study of the
geometry of the curve $C$ and not with $X$.
\end{remark}

To conclude this section, we would like to present two important examples of correspondences with null trace on the $k$-fold symmetric product, which
shall be involved in the proofs of the following sections.

\begin{example}\label{example CORRESPONDENCE of RATIONAL MAP}
For any dominant rational map ${F\colon C^{(k)}\dashrightarrow \mathbb{P}^k}$ of degree $d$, its graph
\begin{displaymath}
\Gamma :=\overline{\left\{(y,P)\in \mathbb{P}^k\times C^{(k)}\mid F(P)=y \right\}},
\end{displaymath}
is a $(d,1)$-correspondence on ${\mathbb{P}^k\times C^{(k)}}$ with null trace. To see this fact, notice that the fiber $F^{-1}(y)$ over a generic point
${y\in \mathbb{P}^k}$ is given by $d$ distinct points ${P_1,\ldots,P_d\in C^{(k)}}$. Hence $\Gamma$ is a reduced variety and the projection
${\pi_1\colon\Gamma\longrightarrow \mathbb{P}^k}$ is a generically finite dominant morphism of degree $d$.
Moreover, we have that the map ${\gamma\colon U\subset\mathbb{P}^k\longrightarrow \left(C^{(k)}\right)^{(d)}}$ introduced at the beginning of this section is a rational map between smooth projective varieties.
Therefore the indeterminacy locus can be resolved to a codimension 2 subvariety of $\mathbb{P}^k$.
Being ${H^{k,0}(\mathbb{P}^k)}$ trivial, we then have that the trace map $Tr_{\gamma}$ is null.
Thus $\Gamma$ is a $(d,1)$-correspondence on ${\mathbb{P}^k\times C^{(k)}}$ with null trace.
We note that this fact is still true for any smooth $n$-dimensional projective variety admitting a dominant rational map of degree $d$ on $\mathbb{P}^n$.
\end{example}

\begin{example}\label{example CORRESPONDENCE of MOVING CURVES COVERING C(k)}
Let $T$ be a ${(k-1)}$-dimensional smooth variety and let ${\mathcal{E}=\{E_t\}_{t\in T}}$ be a family of curves covering $C^{(k)}$ such
that the generic member $E_t$ is an irreducible $d$-gonal curve, i.e. the normalization of $E_t$ has a base-point-free $g^1_d$. As in \cite[proof of Corollary 1.7]{LP}, we want to define a correspondence ${\Gamma\subset Y\times C^{(k)}}$ with null trace and degree $d$ for an appropriate ruled variety $Y$.

To this aim, we think the variety $C^{(k)}$ embedded in some projective space and for any $t\in T$, we denote by $H_t$ the Hilbert scheme of curves on $C^{(k)}$ containing the point $e_t$ representing $E_t$. We note that if $t$ is generic, then $\dim H_t \geq k-1$. Indeed, if $\dim H_t< k-1$ for any $t$, then ${\displaystyle\bigcup_{t\in T} H_t=\bigcup_{\deg \widetilde{E}_t,g(\widetilde{E}_t)} Hilb_{\deg \widetilde{E}_t,g(\widetilde{E}_t)}}$ would be the union of countably many schemes of dimension $\leq k-2$, and the curves parametrized over it would cover the $k$-dimensional variety $C^{(k)}$, a contradiction.
Given a generic point $t\in T$, let $H\subset H_t$ be a $(k-1)$-dimensional subvariety containing $e_t$.
We recall that the Hilbert scheme is a fine moduli space, so we can consider the universal family ${\mathcal{U}:=\left\{(e_{\tau},P)\mid e_{\tau}\in H \textrm{ and } P\in E_{\tau}\right\}}$ over $H$ and -~up to shrink $H$~- the family $\widetilde{\mathcal{U}}$ of the normalized curves. Furthermore, by making a base change
\begin{displaymath}
\xymatrix{ \mathcal{F}\ar[d]_{\rho}\ar[r] & \widetilde{\mathcal{U}}\ar[d]^{pr_1}\\ B\ar[r] & H\\}
\end{displaymath}
we can assume to have a map ${\mu\colon \mathcal{F}\longrightarrow \mathbb{P}^1\times B}$ such that the diagram
\begin{displaymath}
\xymatrix{ \mathcal{F}\ar[dr]_{\rho}\ar[r]^{\mu} & \mathbb{P}^1\times B\ar[d]^{pr_2}\\  & B\\}
\end{displaymath}
is commutative and for any $b\in B$, the restriction ${\mu_b\colon \mathcal{F}_b\longrightarrow \mathbb{P}^1}$ is the given $g^1_d$ on the curve ${\mathcal{F}_b}:=\rho^{-1}(b)$. We note that any fiber of $\mathcal{F}$ is the normalization $\nu_b\colon \mathcal{F}_b\longrightarrow \mathcal{F}'_b$ of a curve lying on $C^{(k)}$. We then define the $k$-dimensional varieties $Y:=\mathbb{P}^1\times B$ and
\begin{displaymath}
\Gamma :=\overline{\left\{\big( (z,b) , P \big)\in Y\times C^{(k)}\mid P\in (\mathcal{F}'_b)_{reg} \textrm{ and } \mu_b\circ \nu^{-1}_b(P)=z \right\}}.
\end{displaymath}

We claim that ${\Gamma\subset Y\times C^{(k)}}$ is a correspondence of degree $d$ with null trace.
Since the $\mathcal{F}_b$'s are the normalizations of curves covering an open subset of $C^{(k)}$ and the map $\mu_b$ is a $g^1_d$, we have that both the projections ${\pi_1\colon \Gamma\longrightarrow Y}$ and ${\pi_2\colon \Gamma\longrightarrow C^{(k)}}$ are dominant morphisms, with ${\deg \pi_1=d}$.
Furthermore, the map $\pi_2$ is generically finite as well: if there exist infinitely many curves $\mathcal{F}'_b$ passing through the generic point ${P\in C^{(k)}}$, then $B$ would be at least a $k$-dimensional variety.
Finally, the space ${H^{k,0}(Y)}$ is trivial because $Y$ is a ruled projective variety and -~by arguing as in the previous example~- we deduce that the correspondence $\Gamma$ has null trace.
\end{example}

\bigskip
\section{Degree of gonality}\label{section DEGREE OF GONALITY}

Let $C$ be a smooth complex projective curve of genus $g\geq 0$.
In this section we deal with the gonality of moving curves on the second symmetric product $C^{(2)}$ and we compute the degree of gonality of this
surface, which is the positive integer we defined to be
\begin{displaymath}
\degree_o(C^{(2)}):=\min\left\{d\in \mathbb{N}\,\left|\,\begin{array}{l} \textrm{there exists a family } \mathcal{E}=\left\{E_t\right\}_{t\in T} \\
\textrm{covering }C^{(2)} \textrm{ whose generic member is} \\ \textrm{an irreducible $d$-gonal curve} \end{array}\right. \right\}\,.
\end{displaymath}

\smallskip
The degree of gonality of $C^{(2)}$ in cases of low genera is easily given. When $C$ is a rational curve, then ${C^{(2)}\cong \mathbb{P}^2}$ and hence
${\degree_o(C^{(2)})=1}$.
On the other hand, if $C$ is supposed to be an elliptic curve, then the fibers of the Abel map ${C^{(2)}\longrightarrow J(C)\cong C}$ are isomorphic to
$\mathbb{P}^1$ and the second symmetric product of $C$ is birational to ${C\times \mathbb{P}^1}$, thus ${\degree_o(C^{(2)})=\degree_o(C\times
\mathbb{P}^1)=1}$.

\smallskip
We note that for any $g\geq 0$, the surface $C^{(2)}$ is covered by the family ${\mathcal{C}=\left\{C_p\right\}_{p\in C}}$ of curves parametrized over
$C$, where ${C_p:=C+p=\left\{p+q\,|\,q\in C\right\}}$.
Clearly, any $C_p$ is isomorphic to $C$ itself.
Therefore the degree of gonality of the second symmetric product of a curve of genus $g\geq 0$ is such that
\begin{equation}\label{equation DEGREE OF GONALITY}
\degree_o(C^{(2)})\leq \gon (C).
\end{equation}
\noindent In particular, since the only rational curve lying on the second symmetric product of a hyperelliptic curve of genus $g=2$ is the fiber of
the $g^1_2$ via the Abel map ${u\colon C^{(2)}\rightarrow J(C)}$, we have that $\degree_o(C^{(2)})=2$.

\smallskip
When the curve $C$ has genus $g\geq 3$, Theorem \ref{theorem INTRO GONALITY OF MOVING CURVES} assures that the gonality of moving curves on $C^{(2)}$ must be greater or equal than $\gon(C)$. Hence Theorem \ref{theorem INTRO DEGREE OF GONALITY} follows straightforwardly from the latter result and inequality (\ref{equation DEGREE OF GONALITY}). We then resolved the problem of computing the degree of irrationality on second
symmetric products of curves, that is $\degree_o(C^{(2)})=\gon (C)$ for any $C$ of genus $g\geq 3$.

\begin{remark}
Let $C$ be a smooth projective curve of genus ${g\geq 0}$ and let $C^{(k)}$ be its $k$-fold symmetric product, with ${k\geq 3}$.

When ${k>g}$, we have that $C^{(k)}$ is birational to ${J(C)\times \mathbb{P}^{k-g}}$ by Abel's theorem.
Thus $C^{(k)}$ is covered by a family of rational curves and hence the degree of gonality is ${\degree_o(C^{(k)})=1}$.

On the other hand, when ${k\leq g}$, the $k$-fold symmetric product of $C$ is covered by the family ${\mathcal{C}=\left\{C_P\right\}_{P\in C^{(k-1)}}}$ of curves parametrized over $C^{(k-1)}$, where ${C_P:=\left\{P+q\,|\,q\in C\right\}}$.
Since any such a curve is isomorphic to $C$, we still have the inequality ${\degree_o(C^{(k)})\leq \gon(C)}$.
In particular, it seems natural to conjecture that this bound is actually an equality, but the techniques we used above do not work to prove this
fact.
\end{remark}

Let us now prove Theorem \ref{theorem INTRO GONALITY OF MOVING CURVES}.
The argument of the proof is essentially based on Theorem \ref{theorem CORRESPONDENCES ON C(k)} and Abel's theorem.

\subsection*{Proof of Theorem \ref{theorem INTRO GONALITY OF MOVING CURVES}}
Notice that if $C$ is a hyperelliptic curve of genus ${g\geq 3}$, the only rational curve lying on $C^{(2)}$ is the fiber of the $g^1_2$ via the Abel map ${u\colon C^{(2)}\rightarrow J(C)}$.
Therefore the gonality of the generic curve $E_t$ must be ${d\geq 2= \gon (C)}$ and the assertion follows. Then we assume hereafter that $C$ is
non-hyperelliptic.
Moreover, aiming for a contradiction we suppose that ${d<\gon(C)}$ and we proceed by steps.

\smallskip
\step{1}{Correspondence on $C^{(2)}$} As $C^{(2)}$ is two-dimensional and $\mathcal{E}$ is a family of curves - up to restrict $\mathcal{E}$ to a
subvariety of $T$ - we can assume that $T$ has dimension one. Following Example \ref{example CORRESPONDENCE of MOVING CURVES COVERING C(k)} we can then construct a family ${\mathcal{F}\stackrel{\rho}{\longrightarrow} B}$ of smooth $d$-gonal curves such that the diagram
\begin{displaymath}
\xymatrix{ \mathcal{F}\ar[dr]_{\rho}\ar[r]^{\mu} & \mathbb{P}^1\times B\ar[d]^{pr_2}\\  & B\\}
\end{displaymath}
is commutative and for any $b\in B$, the restriction ${\mu_b\colon \mathcal{F}_b\longrightarrow \mathbb{P}^1}$ is the given $g^1_d$ on the curve ${\mathcal{F}_b:=\rho^{-1}(b)}$.
In particular, any such a curve is the normalization $\nu_b\colon \mathcal{F}_b\longrightarrow \mathcal{F}'_b$ of a curve lying on $C^{(2)}$, which generically is one of the $E_t$'s.
Then we set $Y:=\mathbb{P}^1\times B$ and we have that the surface
\begin{displaymath}
\Gamma :=\overline{\left\{\big( (z,b) , P \big)\in Y\times C^{(k)}\mid P\in (\mathcal{F}'_b)_{reg} \textrm{ and } \mu_b\circ \nu^{-1}_b(P)=z \right\}}
\end{displaymath}
is a correspondence of degree $d$ on ${Y\times C^{(k)}}$ with null trace.

As usual, let ${\pi_1\colon \Gamma\longrightarrow Y}$ be the restriction of the first projection map and for a very general point ${(z,b)\in Y}$, let
${\pi_1^{-1}(z,b)=\left\{\big( (z,b) , P_i \big)\in Y\times C^{(2)}\,|\,i=1,\ldots,d\right\}}$ be its fiber, with ${P_i=p_{2i-1}+p_{2i}}$.
Moreover, let ${D=D_{(z,b)}\in Div(C)}$ be the effective divisor given by
\begin{equation}\label{equation DIVISOR}
D:=p_1+\ldots+p_{2d}=\sum_{j=1}^m n_j q_j
\end{equation}
for some positive integers ${n_j=\mult_{q_j}(D)}$, where the $q_j$'s are assumed to be distinct points of $C$.

\smallskip
\step{2}{$n_j=1$ for all $j$} We suppose that $n_j=1$ for any ${1\leq j\leq m}$, that is $m=2d$ and the points defining $D$ are all distinct. Let
${\phi\colon C\longrightarrow \mathbb{P}^{g-1}}$ be the canonical embedding of $C$ and let $\overline{\phi(D)}$ be the linear span of the points
$\phi(p_i)'s$ in $\mathbb{P}^{g-1}$. As $\Gamma$ is a correspondence of degree $d$ on $C^{(2)}$ with null trace, by Theorem \ref{theorem
CORRESPONDENCES ON C(k)} we have that $\dim \overline{\phi(D)}\leq d-1$.

If $C$ is a non-hyperelliptic curve of genus $g=3$, we have that $C^{(2)}$ is non covered by a family of rational curves and $\gon(C)=3$.
Hence $d=2$ and the curves of the family $\mathcal{F}$ are hyperelliptic.
Moreover, the linear span ${\overline{\phi(D)}\subset \mathbb{P}^2}$ of the points $\phi (p_1),\ldots ,\phi (p_4)$ is a line, that is $D$ is a
canonical divisor on $C$.
Let ${\iota\colon C^{(2)}\longrightarrow C^{(2)}}$ be the involution ${p+q\longmapsto K_{C}-p-q}$ sending a point to the residual of the canonical
system.
We note that the map $\iota_{b}$ induced on the generic curve $\mathcal{F}_b$ of the family $\mathcal{F}$ is the hyperelliptic involution.
Thus the quotient surface $C^{(2)}/\langle\iota\rangle$ is covered by a family of rational curves, but this is impossible because the latter surface is of general type.
To see this fact it suffices to observe that $C^{(2)}$ is a surface of general type and that the map ${C^{(2)}\longrightarrow \mathbb{P}^{{g \choose
2}-1}}$ induced by the canonical linear system $|K_{C^{(2)}}|$ factors through $\iota$, because the Gauss map $\mathcal{G}$ in (\ref{equation GAUSS
MAP}) does.

On the other hand, let us assume $g\geq 4$.
By the geometric version of Riemann-Roch theorem we have
\begin{equation*}
\dim |D|=\deg D-1-\dim \overline{\phi(D)}\geq 2d-1-(d-1)=d=\frac{\deg D}{2}.
\end{equation*}
Therefore we have that either $D$ is zero, $D$ is a canonical divisor or $C$ is hyperelliptic by Clifford's theorem (cf. \cite[p. 107]{ACGH}).
Notice that $C$ is assumed to be non-hyperelliptic and ${0<d<\gon(C)\leq\left[\frac{g+3}{2}\right]}$.
Hence ${0<\deg D=2d<2g-2}$ for any ${g\geq 4}$ and we have a contradiction.
Thus the points ${p_1,\ldots,p_{2d}}$ can not be distinct.

\smallskip
\step{3}{$n_j>1$ for some $j$} Let us then assume that the points ${p_1,\ldots,p_{2d}}$ are not distinct, i.e. the integers $n_j$'s are not all equal to 1.
For any ${k=1,\ldots,2d}$ let us consider the set ${Q_{k}:=\left\{q_j\in \Supp D\,|\,n_j=k\right\}}$ of the points of $D$ such that ${\mult_{q_j}(D)=k}$.
Notice that the cardinality of any $Q_{k}$ is at most ${\left[\frac{2d}{k}\right]}$.
As the $n_j$'s are not all equal to 1, there exist some ${k>1}$ such that the corresponding sets $Q_{k}$ are not empty. Let $h>1$ be the minimum of these integers and -~without loss of generality~- suppose ${Q_{h}=\{q_1,\ldots,q_s\}}$, where ${s\leq\left[\frac{2d}{h}\right]}$ is the cardinality of $Q_{h}$.

Since $Y$ is connected, the fibers of $\pi_1$ over generic points of $Y$ have the same configuration, i.e. the cardinality of any set $Q_{h}$ is
constant as we vary the point $(z,b)$ on a suitable open set $U\subset Y$.
Thus we may define a rational map $\xi\colon Y \dashrightarrow C^{(s)}$ sending a generic point $(z,b)\in Y=\mathbb{P}^1\times B$ to the effective
divisor $q_1+\ldots+q_s\in C^{(s)}$. For a very general $b\in B$, let
\begin{displaymath}
\begin{array}{cccl}
\xi_b\colon & \mathbb{P}^1\times\{b\} & \longrightarrow & C^{(s)}\\
 & (z,b) & \longmapsto & q_1+\ldots+q_s
\end{array}
\end{displaymath}
be the restriction of $\xi$ to the rational curve ${\mathbb{P}^1\times\{b\}\subset Y}$ and let us consider the composition with the Abel map
\begin{displaymath}
 \mathbb{P}^1\times\{b\} \stackrel{\xi_b}{\longrightarrow} C^{(s)} \longrightarrow J(C)\,.
\end{displaymath}
As $\mathbb{P}^1\times\{b\}$ is a rational curve mapping into a Jacobian variety, the latter map is constant.
Hence by Abel's theorem, either ${|q_1+\ldots+q_s|}$ is a complete linear series of degree $s$ and dimension at least 1, or $\xi_b$ is a constant map.
Being $h>1$ and ${s\leq\left[\frac{2d}{h}\right]}$, we have ${s\leq d< \gon(C)}$. Then ${|q_1+\ldots+q_s|}$ can not be such a linear series.

Therefore the map $\xi_b$ must be constant.
By the construction of $\xi_b$, this fact means that for any $z\in \mathbb{P}^1$ the divisor $D=D_{(z,b)}$ -~defined in (\ref{equation DIVISOR}) by the fiber $\pi_1^{-1}(z,b)$~- must contain all the points $q_1,\ldots,q_s$, that are now fixed.
We recall that ${\pi_1^{-1}(z,b)}$ is given by the points ${((z,b),P_i)\in Y\times C^{(2)}}$ such that ${P_i\in \mathcal{F}'_b}$ and ${\mu_b\circ \nu_b^{-1}(P_i)=z}$, where ${P_i=p_{2i-1}+p_{2i}}$.
Hence one of the $P_i$'s must lie on the curve $C+q_1$, one on $C+q_2$ and so on.
As we vary $z$ on $\mathbb{P}^1$, the $P_i$'s must vary on $\mathcal{F}'_b$, but the latter condition must hold.
It follows that the curve $\mathcal{F}'_b$ must have at least $s$ irreducible components ${\mathcal{F}'_{b1},\ldots,\mathcal{F}'_{bs}}$ such that ${\mathcal{F}'_{bj}\subset C+q_j}$ for ${1\leq j\leq s}$. We recall that for the generic ${b\in B}$, the curve $\mathcal{F}'_b$ coincides with a generic element $E_{t}$ of the family $\mathcal{E}$.
Since both $E_{t}$ and ${C+q_j}$ are irreducible curves, we deduce ${s=1}$ and ${E_t=C+q_1}$.
Then we get a contradiction because ${C+q_1\cong C}$ and hence ${d=\gon(E_t)=\gon(C)}$.
Thus we conclude that the gonality $d$ of the generic $E_t$ is ${d\geq \gon(C)}$.

\smallskip
\step{4}{$d=\gon(C)$} To conclude the proof of the statement, let us suppose that $C$ is a curve of genus ${g\geq 6}$ with
${Aut(C)=\left\{Id_C\right\}}$ and ${d=\gon(\widetilde{E}_t)=\gon(C)}$. We want to prove that the generic $E_t$ and $C$ are isomorphic.

To this aim, let us consider the correspondence ${\Gamma\subset Y\times C^{(2)}}$ defined above and for a generic point ${(z,b)\in Y}$, let
${\pi_1^{-1}(z,b)=\left\{\big( (z,b) , P_i \big)\right\}}$ be its fiber, with ${P_i=p_{2i-1}+p_{2i}}$.
By arguing as in Step 1 we deduce that if the $p_i$'s can not be distinct.
If they were distinct, then ${\dim |D|=\frac{\deg D}{2}}$ and -~by Clifford's theorem~- we would have that either $D$ is zero, $D$ is a canonical
divisor or $C$ is hyperelliptic.
We note that the assumption ${Aut(C)=\left\{Id_C\right\}}$ implies that $C$ is non-hyperelliptic.
Moreover, the degree of $D$ is positive and
\begin{displaymath}
\deg D=2d=2 \gon(C)\leq 2\left[\frac{g+3}{2}\right]<2g-2\quad \textrm{for any } g\geq 6.
\end{displaymath}
Hence the divisor $D$ is neither zero nor canonical and we have a contradiction.

Then we follow the argument of Step 3 and for the generic ${b\in B}$ we may define the map ${\xi_b\colon \mathbb{P}^1\times\{b\}\longrightarrow
C^{(s)}}$, with ${s\leq d}$.

If ${s<d=\gon (C)}$, the only possible choice is ${s=1}$ because of the irreducibility of $E_t$.
Hence the generic $E_t$ and $C$ turn out to be isomorphic.

On the other hand suppose that ${s=d=\gon(C)}$.
Then ${h=2}$ and the divisor $D=D_{(z,b)}$ in (\ref{equation DIVISOR}) has the form ${D=2(q_1+\ldots+q_d)}$.
Since the point ${(z,b)\in Y}$ is generic and the projections ${\pi_1\colon \Gamma \longrightarrow Y}$, ${\pi_2\colon \Gamma \longrightarrow C^{(2)}}$ are generically finite dominant morphism, we can assume that none of the $P_i$'s lie on the diagonal curve ${\Delta:=\{p+p\mid p\in C\}\subset C^{(2)}}$: indeed, if the fiber of $\pi_1$ over the generic point met the curve ${\pi_2^{-1}(\Delta)\subset \Gamma}$, such a curve would dominate the surface $Y$.
Moreover, let $E_t$ be the generic element of the family $\mathcal{E}$ corresponding to the curve $\mathcal{F}'_b$.
Thus -~without loss of generality~- the points ${P_i\in E_t\subset C^{(2)}}$ are given by
\begin{displaymath}
P_1=q_1+q_2,\,P_2=q_2+q_3,\ldots,\,P_{d-1}=q_{d-1}+q_d\,\textrm{ and }\,P_d=q_d+q_1.
\end{displaymath}
Furthermore, since the map ${\xi_b\colon \mathbb{P}^1\times\{b\}\longrightarrow C^{(d)}}$ is non constant and $d=\gon(C)$, we have that ${|q_1+\ldots+q_d|}$ is a base-point-free $g^1_d$ on $C$. This imply that ${(E_t\cdot C+q_j)=2}$ for any $j$; indeed ${(E_t\cdot C+q_j)\geq 2}$ because two of the $P_i$'s lie on ${C+q_j}$, and if there existed another point ${p+q_j\in E_t}$, then it would lie on the support of a divisor $D_{(\overline{z},b)}$ of the $g^1_d$ for some ${\overline{z}\in \mathbb{P}^1}$, thus $q_j$ would be a base point. Hence ${(E_t\cdot C+q)=2}$ for any ${q\in C}$ because the curves ${C+q}$'s are all numerically equivalent. We then distinguish two cases.

Suppose that $d=2n$ is even and let us show that this situation can not occur. So let us consider the permutation ${\sigma\in S_d}$ given by ${\sigma(j)=j+n}$ mod. $d$ for any ${1\leq j\leq d}$. Then $\sigma$ induces an involution $\alpha_{\sigma}$ on ${\{q_1,\ldots,q_d\}}$ sending a point $q_j$ to the point $q_{\sigma(j)}$. In other words, we can think the $q_j$'s as the vertices of a convex polygon whose sides correspond to the $P_i$'s (e.g. the side $P_1$ is the one joining $q_1$ and $q_2$), and the involution above sends any $q_j$ to the opposite vertex. In particular, this point of view shows that $\alpha_{\sigma}$ depends only from the configuration of the fiber over $(z,b)$ and it can be defined independently from the choice of the indices of the $q_j$'s. As $(E_t\cdot C+q)=2$ for any ${q\in C}$ and the fiber of $\pi_1$ over $(z,b)$ varies holomorphically as we vary ${z\in \mathbb{P}^1}$, we can extend the involution above to an automorphism ${\alpha\colon C\longrightarrow C}$. Since $\alpha$ is not the identity on $C$ and ${Aut(C)=\left\{Id_C\right\}}$, we have a contradiction.

Finally, let us assume that ${d=2n+1}$ is odd and let us show that $E_t\cong C$. We can define a one-to-one map from ${\{q_1,\ldots,q_d\}}$ to the set of the $P_i$'s by sending a point $q_j\in C$ to the point $q_{j+n}+q_{j+n+1}\in E_t$, where the indices are taken mod. $d$ (i.e. such a map associates to any vertex of the polygon the opposite side). By fixing $b\in B$ and varying ${z\in \mathbb{P}^1}$, we then have an isomorphism between $C$ and $E_t$ as claimed.

Thus Theorem \ref{theorem INTRO GONALITY OF MOVING CURVES} is now proved.

\bigskip
\section{Degree of irrationality}\label{section DEGREE OF IRRATIONALITY}

Let $C$ be a smooth complex projective curve of genus $g\geq 0$.
We want to study the degree of irrationality of the surface $C^{(2)}$,
\begin{displaymath}
\degree_r(C^{(2)}):=\min\left\{d\in \mathbb{N}\,\left|\,\begin{array}{l} \textrm{there exists a dominant rational }\\  \textrm{map } F\colon
C^{(2)}\dashrightarrow \mathbb{P}^2 \textrm{ of degree } d\end{array}\right. \right\}\,,
\end{displaymath}
in dependence both on the genus and on the gonality of $C$.

\smallskip
When $C$ is either a rational or an elliptic curve, the problem of determining the degree of irrationality of $C^{(2)}$ is totally understood. Namely,
if $C$ is rational, then $C^{(2)}$ is isomorphic to $\mathbb{P}^2$.
Hence the second symmetric product is a rational surface and ${\degree_r(C^{(2)})=1}$.

On the other hand, let us suppose that $g=1$. By Abel's theorem $C^{(2)}$ is birational to the non-rational surface $C\times \mathbb{P}^1$.
The curve $C$ admits a double covering ${f\colon C\longrightarrow\mathbb{P}^1}$, therefore we may define the degree two map ${f\times
Id_{\mathbb{P}^1}\colon C\times \mathbb{P}^1\longrightarrow\mathbb{P}^1\times \mathbb{P}^1}$.
Finally, being ${\mathbb{P}^1\times \mathbb{P}^1}$ and $\mathbb{P}^2$ birational surfaces, we conclude ${\degree_r(C^{(2)})=\degree_r(C\times
\mathbb{P}^1)=2}$.

\smallskip
When $C$ is a smooth complex projective curve of genus ${g\geq 2}$, the problem of computing the degree of irrationality of $C^{(2)}$ is still open.
As a consequence of the main result in \cite{AP}, we have the following result providing a lower bound.

\begin{proposition}\label{proposition ALZATI PIROLA}
Let $C^{(k)}$ be the $k$-fold symmetric product of a smooth curve $C$ of genus ${g\geq k\geq 2}$.
Then ${\degree_r(C^{(k)})\geq k+1}$.
\begin{proof}
Let ${F\colon C^{(k)}\dashrightarrow \mathbb{P}^k}$ be a dominant rational map of degree $d$.
By \cite[Theorem 3.4]{AP}, we have that ${d\left(\hl(\mathbb{P}^k)+1\right)\geq \hl(C^{(k)})+1}$, where $\hl(X)$ denotes the length of the graded
algebra ${H^{1,0}(X)\oplus\ldots\oplus H^{k,0}(X)}$, that is the maximum integer $r$ such that there exist homogeneous elements
${\omega_1,\ldots,\omega_r}$ with ${\omega_1\wedge\ldots\wedge\omega_r\neq 0}$.
Since ${\hl(\mathbb{P}^k)=0}$ and ${\hl(C^{(k)})=k}$ by (\ref{equation CANONICAL LINEAR SERIES ON C(k)}), we conclude that ${d\geq k+1}$ as claimed.
\end{proof}
\end{proposition}

In particular, we have that ${\degree_r(C^{(2)})\geq 3}$ for any curve of genus ${g\geq 2}$. Moreover, this estimate
turns out to be sharp when the curve has genus two. Indeed, in \cite[Theorem 3.1]{TY} has been presented an example of a genus two curve $C'$ whose Jacobian satisfies ${\degree_r(J(C'))= 3}$. As $C'^{(2)}$ maps birationally on $J(C')$ we have ${\degree_r(C'^{(2)})=3}$ as well.

\smallskip
On the other hand, to provide upper bounds on $\degree_r(C^{(2)})$ we have to present dominant rational maps ${C^{(2)}\dashrightarrow \mathbb{P}^2}$.
In the examples below we exploit the existence of linear series on $C$ in order to produce such maps. As a consequence, we achieve the upper bound on $\degree_r(C^{(2)})$ stated in Proposition \ref{proposition INTRO UPPER BOUND ON DEG IRR C(2)}.
\begin{example}\label{example SQUARE OF GONALITY}
Let ${f\colon C\longrightarrow \mathbb{P}^1}$ be a morphism of degree $d$.
Then it is always possible to define the dominant morphism ${F\colon C^{(2)}\longrightarrow \left(\mathbb{P}^1\right)^{(2)}\cong\mathbb{P}^2}$ of
degree $d^2$ given by ${p+q\longmapsto f(p)+f(q)}$.
\end{example}
\begin{example}\label{example G2d}
Suppose that $C$ admits a birational mapping ${f\colon C\longrightarrow \mathbb{P}^2}$ onto a non-degenerate curve of degree $d$.
Hence we may define a dominant rational map ${F\colon C^{(2)}\dashrightarrow \mathbb{G}(1,2)\cong\mathbb{P}^2}$ of degree ${d \choose 2}$ by sending a
point ${p+q\in C^{(2)}}$ to the line of $\mathbb{P}^2$ passing through $f(p)$ and $f(q)$.
\end{example}
\begin{example}
Let ${f\colon C\longrightarrow \mathbb{P}^3}$ be a birational map onto a non-degenerate curve of degree $d$.
Consider a plane ${H\subset \mathbb{P}^3}$ and let ${F\colon C^{(2)}\dashrightarrow H\cong\mathbb{P}^2}$ be the dominant rational map sending a point
${p+q\in C^{(2)}}$ to the intersection of $H$ with the line of $\mathbb{P}^3$ passing through $f(p)$ and $f(q)$.
We note that the degree of $F$ is ${\frac{(d-1)(d-2)}{2}-g}$.
To see this fact, notice that the degree of $F$ is the number of bi-secant line to $f(C)$ passing through a general point ${y\in H}$, and consider the
projection ${\pi_y\colon f(C)\longrightarrow \mathbb{P}^2}$.
As the number of such bi-secant lines equals the number of nodes of the image ${C':=(\pi_y\circ f)(C)}$, and $C'$ is a curve of degree $d$ on
$\mathbb{P}^2$, we conclude that ${\deg F=p_a(C')-g(C')=\frac{(d-1)(d-2)}{2}-g}$.
\end{example}

If $C$ is assumed to be hyperelliptic, by Propositions \ref{proposition INTRO UPPER BOUND ON DEG IRR C(2)} and \ref{proposition ALZATI PIROLA} we have that $\degree_r(C^{(2)})$ is either $3$ or $4$. We mentioned above an example of hyperelliptic curve of genus two with $\degree_r(C^{(2)})=3$. When the genus of $C$ is ${g\geq 4}$ this is no longer possible and Theorem \ref{theorem INTRO DEGIRR HYPERELLIPTIC} asserts that the degree of irrationality of hyperelliptic curves is exactly $4$. In particular, the map on $\mathbb{P}^2$ reaching the minimum degree is the morphism described in Example \ref{example SQUARE OF GONALITY}.

\smallskip
When the curve $C$ is non-hyperelliptic the situation is more subtle and we are not able to compute the precise value of the degree of irrationality of $C^{(2)}$.
However Theorem \ref{theorem INTRO DEGIRR NON-HYPERELLIPTIC} provides several lower bounds in dependence on the genus of the curve.
Notice that for any curve of genus $4\leq g\leq 7$ we have $\degree_r(C^{(2)})\geq g-1$.
As the following examples show, this fails to be true for larger values of $g$ and it seems to happen when $C$ covers certain curves.
\begin{example}\label{example DEGREE OF IRRATIONALITY BIELLIPTIC CURVES}
For an integer $d\geq 2$, let $C$ be a non-hyperelliptic curve of genus ${g\geq 2d^2+2}$ provided of a degree $d$ covering ${f\colon C\longrightarrow
E}$ on an elliptic curve $E$ (a particular case of this setting is given by bielliptic curves of genus greater than 9).
Then we can define the dominant morphism ${C^{(2)}\longrightarrow E^{(2)}}$ of degree $d^2$ sending the point ${p+q\in C^{(2)}}$ to ${f(p)+f(q)\in
E^{(2)}}$.
As we saw at the beginning of this section, ${\degree_r(E^{(2)})=2}$ and there exists a dominant rational map ${E^{(2)}\dashrightarrow \mathbb{P}^2}$
of degree $2$.
Therefore we obtain by composition a dominant rational map ${C^{(2)}\dashrightarrow \mathbb{P}^2}$ of degree $2d^2$.
Thus ${\degree_r(C^{(2)})\leq 2d^2<g-1}$.
\end{example}

\begin{example}
Let $C'$ be the genus two curve of \cite[Theorem 3.1]{TY} we mentioned above and suppose that $C$ is a non-hyperelliptic curve of genus ${g\geq 3d^2+2}$ admitting a degree $d$ covering of $C'$.
Since ${\degree_r(C'^{(2)})=3}$ and $C^{(2)}$ admits a covering of degree $d^2$ of $C'^{(2)}$, it is immediate to check that ${\degree_r(C^{(2)})\leq 3d^2<g-1}$.
\end{example}

On the other hand, when $C$ is assumed to be very general in the moduli space $\mathcal{M}_g$, we have ${\degree_r(C^{(2)})\geq g-1}$ for any genus ${g\geq 4}$, as we stated in Theorem \ref{theorem INTRO DEGIRR VERY GENERAL CURVES}.

\begin{remark}
Let $C$ be a very general curve of genus $g\geq 3$.
The dominant rational map ${C^{(2)}\dashrightarrow \mathbb{P}^2}$ of minimum degree we are able to construct is one of those we used to establish Proposition \ref{proposition INTRO UPPER BOUND ON DEG IRR C(2)}.
As in the proposition, let $\delta_1$ be the gonality of $C$ and for any $m\geq 2$, let $\delta_m$ be the minimum of the integers $d$ such that $C$ admits a birational mapping onto a non-degenerate curve of degree $d$ in $\mathbb{P}^{m}$.
The value of $\delta_m$ can be easily computed using Brill-Noether number and - except for finitely many genera - the map of minimum degree is the one using $g^2_d$'s in Example \ref{example G2d}.
We note further that we can construct analogously a dominant rational map ${C^{(k)}\dashrightarrow \mathbb{P}^k}$ of degree ${\delta_k \choose k}$ by using $g^k_d$'s on $C$.
Then we do not expect the bound in Theorem \ref{theorem INTRO DEGIRR VERY GENERAL CURVES} to be sharp, and we conjecture that - except for finitely many genera - the degree of irrationality of symmetric products of a generic curve $C$ of genus $g$ is ${\degree_r(C^{(k)})={\delta_k \choose k}}$, for any ${1\leq k\leq g-1}$.
\end{remark}

\smallskip
Now, in order to prove the main theorems on this topic, we fix some piece of notation and we state three preliminary lemmas.

By $F\colon C^{(2)}\dashrightarrow \mathbb{P}^2$ we denote hereafter a dominant rational map of minimal degree, that is ${d:=\deg
F=\degree_r(C^{(2)})}$.
Given a point ${y\in \mathbb{P}^2}$, we consider its fiber
\begin{equation}\label{equation FIBER OF y}
F^{-1}(y)=\left\{p_1+p_2,\ldots,p_{2d-1}+p_{2d}\right\}\subset C^{(2)}
\end{equation}
and we define the \emph{divisor} $D_y\in Div(C)$ \emph{associated to} $y$ as
\begin{equation}\label{equation DIVISOR ASSOCIATED TO THE FIBER}
D_y:=p_1+p_2+\ldots+p_{2d-1}+p_{2d}\,.
\end{equation}
Then, by a simple monodromy argument we have the following.
\begin{lemma}\label{lemma MONODROMY OF THE FIBER}
There exists an integer ${1\leq a\leq d}$ such that for a generic point ${y\in \mathbb{P}^2}$, we have
${\mult_{p_j}(D_y)=a}$ for any ${j=1,\ldots,2d}$.\\
In particular, the divisor $D_y$ defined above has the form ${D_y=a\left(q_1+q_2+\ldots+q_m\right)}$, where ${m=\frac{2d}{a}}$ and the $q_j$'s are
distinct point of $C$.
\begin{proof}
Let ${G\colon C\times C\dashrightarrow \mathbb{P}^2}$ be the dominant rational map of degree $2d$ defined as ${G(p,q):=F(p+q)\in \mathbb{P}^2}$. Given
a generic point $y\in \mathbb{P}^2$, let
\begin{displaymath}
G^{-1}(y)=\left\{(p_1,p_2),(p_2,p_1),\ldots,(p_{2d-1},p_{2d}),(p_{2d},p_{2d-1})\right\}\subset C\times C
\end{displaymath}
be its fiber.
Then the divisor $D_y:=p_1+\ldots+p_{2d}$ is uniquely determined by the fiber $G^{-1}(y)$.
Moreover, if $m$ is the number of distinct points of $\{p_1,\ldots,p_{2d}\}$ and we denote by $q_1,\ldots,q_m$ these points, the divisor associated to
$y$ has the form $D_y=\sum_{j=1}^m a_j\, q_j$, for some positive integers $a_j:=\mult_{q_j}(D_y)$.
Therefore we have to prove that ${a_1=\ldots=a_m}$.

As $C\times C$ is a connected surface, the action of the monodromy group ${M\left(G\right)\subset S_{2d}}$ of $G$ is transitive.
Hence it is not possible to distinguish any point of the fiber $G^{-1}(y)$ from another.
Then for any ${(r,s),(v,w)\in G^{-1}(y)}$ we have that ${\mult_{r}(D_y)=\mult_{v}(D_y)}$ and ${\mult_{s}(D_y)=\mult_{w}(D_y)}$.
In particular, we can not distinguish the points $(r,s)$ and $(s,r)$, hence ${\mult_{r}(D_y)=\mult_{s}(D_y)}$.
Thus the divisor $D_y$ must have the same multiplicity at any $p_i$, i.e. there exists an integer ${1\leq a\leq 2d}$ such that ${a=\mult_{p_i}(D_y)}$
for any ${i=1,\ldots,2d}$.
Furthermore, $a$ must divide $2d$ and the number of distinct $p_i$'s is ${m=\frac{2d}{a}}$. Finally, being $y$ generic on $\mathbb{P}^2$, we have that
the number of distinct $p_i$'s is at least $2$. Hence ${m\geq 2}$ and ${a\leq d}$.
\end{proof}
\end{lemma}

The second lemma is a consequence of Abel's theorem.
\begin{lemma}\label{lemma LINEAR SERIES GIVEN BY THE FIBER}
With the notation above, for a generic point ${y\in \mathbb{P}^2}$ with associate divisor ${D_y=a\left(q_1+q_2+\ldots+q_m\right)}$, we have that the
linear series ${|q_1+q_2+\ldots+q_m|}$ is a complete $g^r_m$ on $C$ with ${r\geq 2}$.
Moreover, the integer $a$ is lower than ${d=\deg F}$.
\begin{proof}
Thanks to the previous lemma we are able to define the rational map ${\xi\colon \mathbb{P}^2 \dashrightarrow C^{(m)}}$ sending a generic point ${y\in
\mathbb{P}^2}$ to the effective divisor ${q_1+q_2+\ldots+q_m\in C^{(m)}}$.
As the image of ${y\in \mathbb{P}^2}$ depends on its fiber via the rational dominant map ${F\colon C^{(2)}\dashrightarrow \mathbb{P}^2}$, we have that
$\xi$ is non constant.
Consider the resolution ${\widetilde{\xi}\colon R\longrightarrow C^{(m)}}$ of $\xi$ and the composition with the Abel-Jacobi map
${R\stackrel{\widetilde{\xi}}{\longrightarrow} C^{(m)}\stackrel{u}{\longrightarrow} J(C)}$, where $R$ is a rational surface.
By the universal property of Albanese morphism, the latter map factors through the Albanese variety $\Alb(R)$ of the rational surface $R$, which is
{0-dimensional}.
Hence the composition $u\circ \widetilde{\xi}$ is a constant map.
Being $\xi$ non-constant, by Abel's theorem it follows that for all the generic points ${y\in \mathbb{P}^2}$, the divisors $D_y$ are all linearly
equivalent.
Furthermore, as $y$ vary on a surface, we deduce that the complete linear series ${|q_1+q_2+\ldots+q_m|}$ has dimension ${r\geq 2}$.
To conclude, we recall that ${1\leq a\leq d}$. If $a$ were equal to $d$, then ${m=2}$ and the linear series ${|q_1+q_2|}$ would have degree $2$ and
dimension $2$.
Hence ${a<d}$.
\end{proof}
\end{lemma}

Finally, the third lemma is an immediate consequence of Theorem \ref{theorem CORRESPONDENCES ON C(k)} on correspondences with null trace on symmetric
products of curves.
\begin{lemma}\label{lemma DISTINCT POINTS}
Let $C$ be a non-hyperelliptic curve of genus ${g\geq 5}$ and let ${d=\deg F<g-1}$. Then for a generic point ${y\in \mathbb{P}^2}$, we have that the
points ${p_1,\ldots,p_{2d}\in C}$ in $\mathrm{(\ref{equation FIBER OF y})}$ and $\mathrm{(\ref{equation DIVISOR ASSOCIATED TO THE FIBER})}$ are not
distinct, that is ${a\neq 1}$.
\begin{proof}
Let ${D_y=p_1+\ldots+p_{2d}}$ be the divisor associate to a generic point $y\in \mathbb{P}^2$. By contradiction, suppose that ${p_1,\ldots,p_{2d}}$ are
distinct points of $C$.
Let us consider the graph of the rational map ${F\colon C^{(2)}\dashrightarrow \mathbb{P}^2}$,
\begin{equation}\label{equation CORRESPONDENCE}
\Gamma:=\overline{\left\{(y,p+q)\in \mathbb{P}^2\times C^{(2)}\,|\,F(p+q)=y\right\}},
\end{equation}
which is a correspondence with null trace on $\mathbb{P}^2\times C^{(2)}$ of degree $d=\deg F$ (cf. Example \ref{example CORRESPONDENCE of RATIONAL
MAP}).
Let ${\phi\colon C\longrightarrow \mathbb{P}^{g-1}}$ denote the canonical map of $C$ and let $\overline{\phi(D)}$ be the linear span of the points
$\phi(p_i)$'s in $\mathbb{P}^{g-1}$.
Then ${\dim \overline{\phi(D)}\leq d-1}$ by Theorem \ref{theorem CORRESPONDENCES ON C(k)}.
Thus by the geometric version of Riemann-Roch theorem we have
\begin{equation*}
\dim |D_y|=\deg D_y-1-\dim\overline{\phi(D_y)}\geq 2d-1-(d-1)=d=\frac{\deg D_y}{2}.
\end{equation*}
Therefore by Clifford's theorem we have that either $C$ is hyperelliptic, $D_y$ is zero or $D_y$ is a canonical divisor.
On one hand, the curve $C$ is assumed to be non-hyperelliptic. On the other, we have ${0< d<g-1}$ and hence ${0<\deg D_y<2g-2}$. Thus we have a
contradiction and the assertion follows.
\end{proof}
\end{lemma}

\smallskip
So, let us prove the main results on degree of irrationality of second symmetric products of curves.

\subsection*{Proof of Theorem \ref{theorem INTRO DEGIRR VERY GENERAL CURVES}}
Let $C$ be a very general curve of genus ${g\geq 4}$ and let us prove that ${\degree_r(C^{(2)})\geq g-1}$. When the genus of $C$ is $4$, the assertion
follows from Proposition \ref{proposition ALZATI PIROLA}.
Then let us assume that ${g\geq 5}$ and let ${F\colon C^{(2)}\dashrightarrow \mathbb{P}^2}$ be a dominant rational map of degree
${d=\degree_r(C^{(2)})}$. Aiming for a contradiction we assume $d< g-1$.

Let $y\in \mathbb{P}^2$ be a generic point and let $D_y$ be its associate divisor.
Since $C$ is a non-hyperelliptic curve and $d< g-1$, the lemmas above assure that there exists an integer ${1< a< d}$ such that
$D_y=a\left(q_1+q_2+\ldots+q_m\right)$, where $m=\frac{2d}{a}$ and the $q_j$'s are distinct points of $C$.

We claim that $a\neq 2$.
If $a$ were equal to $2$, we would have $m=d$ and arguing as in Step 4 of the proof of Theorem \ref{theorem INTRO GONALITY OF MOVING CURVES}, we could assume that the fiber over the generic $y\in \mathbb{P}^2$ has the form ${F^{-1}(y)=\{q_1+q_2,q_2+q_3,\ldots,q_d+q_1\}}$.
Hence, by fixing a generic point ${q\in C}$, we would have that for any $p\in C$, it would exist a unique point $p'\in C$ such that $F(p+q)=F(q+p')$. Thus we could define an automorphism ${\alpha_q\colon C\longrightarrow C}$ sending a point ${p\in C}$ to the unique point ${\alpha_q(p)\in C}$ such that ${F(p+q)=F(q+\alpha_q(p))}$. In particular, $\alpha_q$ would not be the identity map and we would have a
contradiction, because the only automorphism of a very general curve is the trivial one.

Then we have ${a\geq 3}$. By Lemma \ref{lemma LINEAR SERIES GIVEN BY THE FIBER}, the linear series ${|q_1+q_2+\ldots+q_m|}$ is a complete $g^r_m$ of
$C$ with ${r\geq 2}$. Therefore the variety $W^r_m(C)$ parametrizing complete linear series of degree $m$ and dimension at least $r$ is non-empty.
We recall that when $C$ is a very general curve, the dimension of $W^r_m(C)$ equals the Brill-Noether number ${\rho(g,r,m):=g-(r+1)(g-m+r)}$.
In particular, ${|q_1+q_2+\ldots+q_m|\in W^2_m(C)}$ and hence ${\rho(g,2,m)\geq 0}$.
It follows that
\begin{displaymath}
m\geq \frac{2g+6}{3}\,
\end{displaymath}
On the other hand, we have $a\geq 3$ and $d<g-1$.
Therefore
\begin{displaymath}
m=\frac{2d}{a}<\frac{2g-2}{3}
\end{displaymath}
and we get a contradiction.

\subsection*{Proof of Theorem \ref{theorem INTRO DEGIRR NON-HYPERELLIPTIC}}
Thanks to Proposition \ref{proposition ALZATI PIROLA}, we have that ${\degree_r(C^{(2)})\geq 3}$ and assertion (i) follows.

As usual, let ${F\colon C^{(2)}\dashrightarrow \mathbb{P}^2}$ be a dominant rational map of degree ${d=\degree_r(C^{(2)})}$ and for a generic point $y\in \mathbb{P}^2$, we consider the associated divisor $D_y$.
By Lemmas \ref{lemma MONODROMY OF THE FIBER} and \ref{lemma LINEAR SERIES GIVEN BY THE FIBER} we have that ${D_y=a\left(q_1+q_2+\ldots+q_m\right)}$, where ${1\leq a< d}$ and the $q_j$'s are distinct points of $C$.
Then we proceed by steps.

\smallskip
\step{1}{ $g\geq 5$ $\Rightarrow$ $d\geq 4$ }
We assume that $C$ has genus ${g\geq 5}$ and we prove that ${\degree_r(C^{(2)})\geq 4}$.
By Proposition \ref{proposition ALZATI PIROLA} we have to check that ${d\neq 3}$.
Aiming for a contradiction, we suppose that ${d=\deg F=3}$.

By Lemma \ref{lemma LINEAR SERIES GIVEN BY THE FIBER}, we have that ${|q_1+q_2+\ldots+q_m|}$ is a complete linear series on $C$ of degree $m$ and dimension ${r\geq 2}$.
As $C$ is non-hyperelliptic and ${g\geq 5}$, Martens' theorem assures that ${\dim W^r_m(C)\leq m-2r-1}$ (see \cite[p.
191]{ACGH}).
As the number of $q_j$'s is ${m=\frac{2d}{a}=\frac{6}{a}}$, we have that $W^r_m(C)$ has non-negative dimension only if ${a=1}$.

Since ${g\geq 5}$ and $d=3$, we have that $d<g-1$ and hence the integer $a$ can not be equal to 1 by Lemma \ref{lemma DISTINCT POINTS}.
Therefore we have a contradiction.
Thus $d\geq 4$ and assertion (ii) follows as a consequence.

\smallskip
\step{2}{ $g\geq 6$ $\Rightarrow$ $d\geq 5$ } We prove that $d_r(C^{(2)})\geq 5$ for any non-hyperelliptic curve $C$ of genus ${g\geq 6}$.
By the previous step, it suffices to see that $C^{(2)}$ does not admit dominant rational maps on $\mathbb{P}^2$ of degree $4$.
By contradiction, let us assume ${d=\deg F=4}$.

The argument is the very same of Step 1.
Thanks to Lemma \ref{lemma LINEAR SERIES GIVEN BY THE FIBER} and Martens' theorem, we deduce ${0\leq \dim W^r_m(C)\leq m-2r-1}$ with ${r\geq 2}$ and ${m=\frac{2d}{a}}$.
Since ${d=4}$, it follows that ${a=1}$, but this situation can not occur by Lemma \ref{lemma DISTINCT POINTS}. Then we have a contradiction and assertion (iii) holds.

\smallskip
\step{3}{ $g\geq 7$ $\Rightarrow$ $d\geq \gon(C)$ }
Suppose that $C$ has genus ${g\geq 7}$ and - by contradiction - assume ${d< \gon(C)}$.
Since ${\gon(C)\leq \left[\frac{g+3}{2}\right]<g-1}$, Lemma \ref{lemma DISTINCT POINTS} guarantees that ${a\geq 2}$ and hence ${m=\frac{2d}{a}\leq d<\gon(C)}$.
On the other hand, ${|q_1+q_2+\ldots+q_m|}$ is a complete linear series on $C$ of degree $m$, thus ${m\geq \gon(C)}$.
Then we have a contradiction.

\smallskip
\step{4}{ $g\geq 7$ $\Rightarrow$ $d\geq 6$ } To conclude, we assume that $C$ has genus ${g\geq 7}$ and we prove that ${\degree_r(C^{(2)})\geq 6}$.
Thanks to Step 2, we have to show that the degree of irrationality of $C^{(2)}$ is different from $5$. Again we argue by contradiction and we suppose ${d=\deg F=5}$.

As above, the inequality ${0\leq dim\,W^r_m(C)\leq m-2r-1}$ holds, where ${r\geq 2}$ and ${m=\frac{10}{a}}$.
In this situation, the only possibilities are ${a=1}$ and ${a=2}$.
The integer $a$ must differ from $1$ by Lemma \ref{lemma DISTINCT POINTS}.
So, let us suppose that $a=2$.
Then $m=5$ and the above inequality implies $r=2$.
In particular, the linear series $|q_1+\ldots+q_5|$ is a complete $g^2_5$ on $C$.
As $m$ is prime, the map $C\longrightarrow \mathbb{P}^2$ defined by the $g^2_5$ is birational onto a non degenerate plane quintic, whose arithmetic genus is $6$.
Hence $g\leq 6$, a contradiction.

Thus assertion (iv) follows from Steps 3 and 4.

\subsection*{Proof of Theorem \ref{theorem INTRO DEGIRR HYPERELLIPTIC}}
Let $C$ be an hyperelliptic curve of genus ${g\geq 4}$ and let us prove that ${\degree_r(C^{(2)})\geq 4}$.
From Propositions \ref{proposition INTRO UPPER BOUND ON DEG IRR C(2)} and \ref{proposition ALZATI PIROLA} we deduce ${3\leq\degree_r(C^{(2)})\leq 4}$.
Aiming for a contradiction, we suppose that there exists a dominant rational map ${F\colon C^{(2)}\dashrightarrow \mathbb{P}^2}$ of degree ${d=3}$.

Let ${y\in \mathbb{P}^2}$ be a generic point, with fiber ${F^{-1}(y)=\left\{p_1+p_2,p_3+p_4,p_5+p_6\right\}}$ and associated divisor
${D_y=p_1+\ldots+p_6}$.
Let ${G\colon C\times C\dashrightarrow \mathbb{P}^2}$ be the map of degree 6 defined as ${G(p,q):=F(p+q)}$, and let
${G^{-1}(y)=\left\{(p_1,p_2),(p_2,p_1),\ldots, (p_6,p_5)\right\}}$ be its fiber over $y$.
By arguing as in the proof of Lemma \ref{lemma MONODROMY OF THE FIBER}, we have that the monodromy groups $M(F)\subset S_3$ and $M(G)\subset S_6$ act transitively on $F^{-1}(y)$ and $G^{-1}(y)$ respectively, since $C^{(2)}$ and $C\times C$ are connected surfaces.
It follows that there is no way to distinguish neither the points of the fiber $F^{-1}(y)$ nor those of $G^{-1}(y)$ by some property varying
continuously as $y$ varies on $\mathbb{P}^{2}$. Thus we can not distinguish the $p_i$'s as well by such a property.

Now, let ${f\colon C\longrightarrow \mathbb{P}^1}$ be $g^1_2$ on $C$ and let ${\iota\colon C\longrightarrow C}$ be the induced hyperelliptic
involution.
We recall that the canonical map ${\phi\colon C\longrightarrow \mathbb{P}^{g-1}}$ is the composition of the double covering $f$ and the Veronese map $\nu_{g-1}\colon \mathbb{P}^{1}\longrightarrow \mathbb{P}^{g-1}$ (see e.g. \cite[Proposition 2.2 p. 204]{Mi}).
Moreover, the image $\phi(C)\subset \mathbb{P}^{g-1}$ is set-theoretically the rational normal curve of degree $g-1$ and the covering $\phi\colon C\longrightarrow \phi(C)$ has degree two.
Then two distinct points $p,q\in C$ has the same image if and only if they are conjugated under the hyperelliptic involution.
Since $y\in \mathbb{P}^2$ is generic, we can assume - without loss of generality - that $p_1$ and $p_2$ are not conjugate under the hyperelliptic involution, that is $\phi(p_1)\neq\phi(p_2)$. As the points of the $F^{-1}(y)$ are indinstinguishable, it follows that ${\phi(p_3)\neq\phi(p_4)}$ and ${\phi(p_5)\neq\phi(p_6)}$ as well.

Consider the correspondence ${\Gamma:=\overline{\left\{(y,p+q)\in \mathbb{P}^2\times C^{(2)}\,|\,F(p+q)=y\right\}}}$ defined as the graph of $F$.
Since $\Gamma\subset \mathbb{P}^2\times C^{(2)}$ has null trace and degree $3$ (cf. Example \ref{example CORRESPONDENCE of RATIONAL MAP}), Theorem \ref{theorem CORRESPONDENCES ON C(k)} assures that the points $\phi(p_1),\ldots,\phi(p_6)$ lie on a plane $\pi\subset \mathbb{P}^{g-1}$. Being $\phi(C)$ a rational normal curve, we have that the $\phi(p_i)$'s consist of at most three distinct points.

Suppose that they are exactly three.
Then the $\phi(p_i)$'s are not collinear, because they lie on $\phi(C)$.
Consider the lines ${l_1:=\overline{\phi(p_1)\phi(p_2)}}$, ${l_2:=\overline{\phi(p_3)\phi(p_4)}}$, ${l_3:=\overline{\phi(p_5)\phi(p_6)}\subset \pi}$ and notice that each of them correspond to a point of $F^{-1}(y)$.
Then we can not distinguish them and hence $l_1$, $l_2$ and $l_3$ are all distinct.
Moreover, they must intersect at a same point $p\in \pi$ (cf. Example \ref{example THREE LINES IN Pn}).
Furthermore, each $\phi(p_i)$ has exactly two the preimages on $C$ because of the monodromy of $G$. Then we can assume - without loss of generality - that ${p\neq \phi(p_1)}$ and ${\phi(p_1)=\phi(p_3)}$.
Thus $l_1$ and $l_2$ must coincide, a contradiction.

So, let us suppose that the $\phi(p_i)$'s consist of two distinct points.
By Lemmas \ref{lemma MONODROMY OF THE FIBER} and \ref{lemma LINEAR SERIES GIVEN BY THE FIBER} there exists ${a=1,2}$ such that the divisor associated to $y$ has the form ${D_y=a(q_1+\ldots+q_m)}$, where ${m=\frac{2d}{a}}$ and the $q_j$'s are distinct point of $C$. If $a=2$ we have $m=3$.
Hence there are two points $q_1,q_2$ mapping on $\phi(p_1)$ and $q_3$ on $\phi(p_2)$, but this situation cannot occur because we are distinguishing points.
On the other hand, suppose that $a=1$ and $m=6$.
As both $\phi(p_1)$ and $\phi(p_2)$ has two preimages on $C$, the $q_j$'s must be at most four distinct points.
Thus we have a contradiction and the assertion of Theorem \ref{theorem INTRO DEGIRR HYPERELLIPTIC} holds.

\bigskip
\section{Bounds on the ample cone of second symmetric products of curves}\label{section NEF CONE}

Let $C$ be a smooth complex projective curve of genus $g$ and let us assume that $C$ is very general in the moduli space $\mathcal{M}_g$.
In this section we apply Theorem \ref{theorem INTRO GONALITY OF MOVING CURVES} to the problem of describing the cone $Nef(C^{(2)})_{\mathbb{R}}$ of nef numerical equivalence classes of $\mathbb{R}$-divisor, and we prove Theorem \ref{theorem INTRO NEW BOUNDS}.
To this aim we firstly recall some basic facts on this topic.

\smallskip
Given a point ${p\in C}$, we define the divisors on $C^{(2)}$ given by ${C_p:=\{p+q\mid q\in C\}}$ and ${\Delta:=\{q+q\mid q\in C\}}$.
Let $x$ and $\delta$ denote their numerical equivalence classes in the N\'eron-Severi group $N^1(C^{(2)})$.
The vector space $N^1(C^{(2)})_{\mathbb{R}}$ of numerical classes of $\mathbb{R}$-divisors is spanned by the classes $x$ and $\frac{\delta}{2}$ (cf.
\cite[p. 359]{ACGH}), where ${x^2=1}$, ${\left(\frac{\delta}{2}\right)^2=1-g}$ and ${\left(x\cdot\frac{\delta}{2}\right)=1}$.
Then we deduce the formula governing the intersection on the N\'eron-Severi space, that is
\begin{equation*}
\left((a+b)x-b\frac{\delta}{2}\right)\cdot\left((c+d)x-d\frac{\delta}{2}\right)=ac-bdg.
\end{equation*}

Since $Nef(C^{(2)})_{\mathbb{R}}$ is a two-dimensional convex cone, it is completely determined by its two boundary rays. The first one is the dual ray
of the diagonal via intersection pairing and it is spanned by the class $(g-1)x-\frac{\delta}{2}$. The other ray is spanned by the class
${\left(\tau(C)+1\right)x-\frac{\delta}{2}}$, where $\tau(C)$ is the real number defined as
\begin{equation*}
\tau(C) :=\inf \left\{t>0\left|\, (t+1)x-\frac{\delta}{2} \textrm{ is ample}\right.\right\}.
\end{equation*}
Thus the problem of describing the nef cone $Nef(C^{(2)})_{\mathbb{R}}$ is reduced to compute $\tau(C)$. Clearly, as the self intersection of an ample
divisor is positive, it follows $\tau(C)\geq \sqrt{g}$.

\begin{remark}
We note that when the genus of $C$ is $g\leq 4$, the problem is totally understood (for details see \cite{CK, K, R}).
On the other hand, there is an important conjecture -~due to Kouvidakis~- governing the case $g\geq 5$. It asserts that $\tau(C)= \sqrt{g}$, i.e. the nef cone is as large as
possible. Such a conjecture has been proved in \cite{CK, K} when the genus $g$ is a perfect square, whereas the problem is still open in the other cases.\\
We recall further that when the genus of $C$ is $g\geq 9$, Kouvidakis' conjecture is implied by Nagata's one on the Seshadri constant at $g$ generic points in $\mathbb{P}^2$, and this fact leads to several bounds on $\tau(C)$ (see for instance \cite{R}).
On the other hand, the best previously known bounds for generic curves of genus ${5\leq g \leq 8}$ are those of \cite[Theorem 1]{B}.
\end{remark}

Moreover, we would like to note that the bounds of \cite[Theorem 1]{B} are ${\tau_5=\frac{9}{4}}$, ${\tau_6= \frac{37}{15}}$, ${\tau_7=\frac{189}{71}}$ and ${\tau_8=\frac{54}{19}}$.
Since ${\frac{32}{13}<\frac{37}{15}}$, ${\frac{77}{29}<\frac{189}{71}}$ and ${\frac{17}{6}<\frac{54}{19}}$, we deduce that Theorem \ref{theorem INTRO NEW BOUNDS} does provide an improvement of the bounds on the ample cone of $C^{(2)}$ when ${6\leq g \leq 8}$.

\smallskip
As we anticipated, the argument to prove Theorem \ref{theorem INTRO NEW BOUNDS} is the very same of \cite[Theorem 1]{B}.
Then let us recall two preliminary results involved in the proof.

Given a smooth complex projective variety $X$ and a nef class $L\in N^1(X)_{\mathbb{R}}$, we define the \emph{Seshadri constant} of $L$ at a point
$y\in X$ as the real number
\begin{equation*}
\epsilon \left( y; X, L \right) := \inf_E \frac{(L\cdot E)}{mult_y E}\,,
\end{equation*}
where the infimum is taken over the irreducible curves $E$ passing through $y$. The following holds (see \cite[Theorem 1.2]{R}).
\begin{theorem}\label{theorem ROSS}
Let $D$ be a smooth curve of genus $g-1$. Let $a,b$ be two positive real numbers such that $\frac{a}{b}\geq \tau(D)$ and for a very general point $y\in D^{(2)}$
\begin{equation*}
\epsilon \left( y; D^{(2)}, (a+b)x-b\frac{\delta}{2} \right) \geq b.
\end{equation*}
If $C$ is a very general curve of genus $g$, then $\tau(C)\leq \frac{a}{b}$.
\end{theorem}

Moreover, we need the following lemma (cf. \cite[Lemma 3]{B} and \cite[Theorem A]{KSS}).
\begin{lemma}\label{LEMMA el + gonality}
Let $X$ be a smooth complex projective surface.
Let $T$ be a smooth variety and consider a family $\{y_t\in E_t\}_{t\in T}$ consisting of a curve $E_t\subset Y$ through a very general point $y_t\in
X$ such that $\mult_{y_t}E_t\geq m$ for any $t\in T$  and for some $m\geq 2$.
If the central fiber $E_0$ is a reduced irreducible curve and the family is non-trivial, then
\begin{equation*}
E_0^2\geq m(m-1)+ \gon(E_0).
\end{equation*}
\end{lemma}

\smallskip
\subsection*{Proof of Theorem \ref{theorem INTRO NEW BOUNDS}}
Assume that $g=6$ and let us prove that ${\tau(C)\leq \frac{32}{13}}$.
To this aim, consider a very general curve $D$ of genus ${g(D)=g-1=5}$ together with its second symmetric product $D^{(2)}$. Let ${a=32}$, ${b=13}$ and consider the numerical equivalence class \begin{equation}\label{equation L}
L:=(a+b)x-b\frac{\delta}{2}\in N^1(D^{(2)}),
\end{equation}
which is nef by \cite[Theorem 1]{B}.
Thanks to Theorem \ref{theorem ROSS} it suffices to prove that the Seshadri constant of $L$ at a generic point ${y\in D^{(2)}}$ is greater or equal than $b$, i.e. there is not a reduced irreducible curve ${E\subset D^{(2)}}$ passing through $y$ such that ${(L\cdot E)/mult_y E<b}$.

Let $\mathcal{F}$ be the set of pairs $(F,z)$ such that ${F\subset D^{(2)}}$ is a reduced irreducible curve, $z\in F$ is a point and ${(L\cdot
F)/\mult_z F<13}$.
Such a set consists of at most countably many algebraic families and $y$ is generic on $D^{(2)}$, thus we have to show that each of these families is discrete (cf. \cite[Section 2]{EL}).

We argue by contradiction and we assume that there exists a family ${\mathcal{E}=\{(y_t\in E_t)\}_{t\in T}}$ such that for any ${t\in T}$, the curve ${E_t\subset D^{(2)}}$ is reduced and irreducible, the point $y_t$ is very general on $D^{(2)}$ and
\begin{equation}\label{equation CONTRADICTION}
\frac{(L\cdot E_t)}{\mult_{y_t} E_t}<b=13.
\end{equation}

We claim that for any $t\in T$, we have
\begin{equation}\label{equation INTERSECTION}
(L\cdot E_t)\geq b.
\end{equation}
Let ${(n+\gamma)x-\gamma\frac{\delta}{2}\in N^1(D^{(2)})}$ be the numerical equivalence class of $E_t$.
Since the class $x$ is ample, we have ${(x\cdot E_t)=n>0}$. Being ${(L\cdot E_t)=an-b\gamma g}$, we then have that (\ref{equation INTERSECTION}) holds when ${\gamma\leq 0}$.\\
So, let us assume ${\gamma>0}$.
Since $\mathcal{E}$ is a family of curves covering $D^{(2)}$, we have $(L\cdot E_t)\geq 0$ (cf. \cite[Lemma 2.2]{R}).
Furthermore, $D^{(2)}$ is a non-fibred surface, hence there are at most finitely many irreducible curves of zero self intersection and numerical class $(n+\gamma)x-\gamma(\delta/2)$. Therefore we can assume ${E_t^2=n^2-(g-1)\gamma^2>0}$, that is ${n\geq \gamma\sqrt{g-1} +1}$.
Notice that ${a\geq b\sqrt{g-1}}$, thus ${(L\cdot E_t)=an-(g-1)b\gamma> b\sqrt{g-1}\left(\gamma\sqrt{g-1} +1\right)-(g-1)b\gamma>b}$
and the claim follows.

\smallskip
By (\ref{equation CONTRADICTION}) and (\ref{equation INTERSECTION}) we deduce that ${\mult_{y_t} E_t>(L\cdot E_t)/b\geq 1}$ for any $t\in T$.
As $E_t$ is a reduced curve, we have that ${\mult_{z} E_t=1}$ for any generic point $z\in E_t$.
Hence ${\mathcal{E}=\{(y_t\in E_t)\}_{t\in T}}$ is a non-trivial family. Without loss of generality, let us assume that the central fiber is such that
\begin{equation*}
m:=\mult_{x_0} E_0\leq \mult_{x_t}E_t
\end{equation*}
for any $t\in T$.
Thanks to Lemma \ref{LEMMA el + gonality} we have that the curve $E_0$ has self intersection ${E_0^2\geq m(m-1)+ \gon(E_0)}$.
Furthermore, Theorem \ref{theorem INTRO GONALITY OF MOVING CURVES} assures that ${\gon(E_0)\geq \gon(D)}$, where ${\gon(D)=\left[\frac{g(D)+3}{2}\right]}$ because $D$ is assumed to be very general in $\mathcal{M}_{g-1}$.
Hence
\begin{equation}\label{equation SELF INTERSECTION}
E_0^2\geq m(m-1)+ \left[\frac{(g-1)+3}{2}\right]= m(m-1)+4.
\end{equation}
Finally, inequality (\ref{equation CONTRADICTION}) leads to ${(L\cdot E_0)\leq bm-1}$.
Thus by Hodge Index Theorem we have
\begin{equation}\label{equation CONCLUSION}
m(m-1)+ 4 \leq E_0^2 \leq \frac{(L\cdot E_0)^2}{L^2}\leq \frac{(13m-1)^2}{179},
\end{equation}
but this is impossible.
Hence we get a contradiction and we proved that ${\tau(C)\leq\frac{32}{13}}$ for any generic curve of genus ${g=6}$.

\smallskip
Now, let us assume that $C$ has genus ${g=7}$ and let $D$ be a very general curve of genus ${g(D)=g-1=6}$.
In order to follow the above argument, we set ${a=77}$ and ${b=29}$.
We just proved that ${\tau(D)\leq \frac{32}{13}}$, hence the class $L$ defined in (\ref{equation L}) is still nef.
Then we can argue as above and we have ${E_0^2\geq m(m-1)+ \gon(E_0)}$.
We recall that ${E_0\subset D^{(2)}}$ is a singular reduced irreducible curve lying on the second symmetric product of the genus six curve $D$.
Therefore $E_0$ is not isomorphic to $D$, and ${\gon(E_0)\geq \gon(D)+1}$ by Theorem \ref{theorem INTRO GONALITY OF MOVING CURVES}.
Then we obtain the analogous of inequality (\ref{equation SELF INTERSECTION}), that is ${E_0^2\geq m(m-1)+ 5}$.
Thus (\ref{equation CONCLUSION}) becomes
\begin{equation*}
m(m-1)+ 5 \leq E_0^2 \leq \frac{(L\cdot E_0)^2}{L^2}\leq \frac{(29m-1)^2}{883},
\end{equation*}
which is still impossible.
Then we have that ${\tau(C)\leq\frac{77}{29}}$ for any generic curve of genus ${g=7}$.

\smallskip
Analogously, let ${g=8}$ and consider a generic curve $D$ of genus ${g(D)=g-1=7}$.
Since ${\frac{17}{6}>\frac{77}{29}\geq\tau(D)}$, we can argue as above by setting ${a=17}$ and ${b=6}$.
Then we have ${E_0^2\geq m(m-1)+ \gon(E_0)}$, where ${\gon(E_0)\geq \gon(D)+1=6}$.
Therefore
\begin{equation*}
m(m-1)+ 6 \leq E_0^2 \leq \frac{(L\cdot E_0)^2}{L^2}\leq \frac{(6m-1)^2}{37},
\end{equation*}
that still lead to a contradiction.
Thus ${\tau(C)\leq\frac{17}{6}}$ for any very general curve of genus ${g=8}$ and the proof ends.

\section*{Acknowledgements}

This is part of the author's Ph.D. thesis at University of Pavia. I am grateful to my supervisor Gian Pietro Pirola for getting me interested in these problems for his patient helpfulness that made this work possible. I would also like to thank Alessandro Ghigi for helpful discussions, Ciro Ciliberto for suggesting some appropriate references and Alberto Alzati for his remarks on lines in special position.
Finally, I would like to thank the referee for pointing out some inaccuracies and for useful suggestions, which improved this work.

\end{document}